\documentclass[10pt,a4paper]{article}
\usepackage[utf8]{inputenc}
\usepackage{amsmath}
\usepackage{bm}
\usepackage{amsfonts}
\usepackage{amssymb}
\usepackage{amsthm}
\usepackage{color}
\usepackage{marvosym}
\definecolor{rot}{rgb}{0.75 , 0.00 , 0.00}
\definecolor{blau}{rgb}{0.00 , 0.00 , 0.5}
\usepackage[colorlinks=true, citecolor=rot, linkcolor=blau, urlcolor=blue]{hyperref}
\theoremstyle{definition}
\newtheorem{defn}{Definition}
\theoremstyle{plain}
\newtheorem{thm}[defn]{Theorem}
\newtheorem{prop}[defn]{Proposition}
\newtheorem{lmm}[defn]{Lemma}
\newtheorem{cor}[defn]{Corollary}
\theoremstyle{remark}
\newtheorem*{rmk}{Remark}

\newtheorem*{ex}{Example}

\newcommand{\vertiii}[1]{{\left\vert\kern-0.25ex\left\vert\kern-0.25ex\left\vert #1
    \right\vert\kern-0.25ex\right\vert\kern-0.25ex\right\vert}}

\makeatletter
\renewcommand*\env@matrix[1][*\c@MaxMatrixCols c]{%
  \hskip -\arraycolsep
  \let\@ifnextchar\new@ifnextchar
  \array{#1}}
\makeatother

\DeclareMathOperator{\UC}{UC}

\DeclareMathOperator{\BUC}{BUC}
\DeclareMathOperator{\VO}{VO_\partial}
\DeclareMathOperator{\VMO}{VMO}

\newcommand{\phol}{P^\lambda}
\newcommand{\pah}{P_{\text{ah}}^\lambda}
\newcommand{\pph}{P_{\text{ph}}^\lambda}
\newcommand{\pcons}{P_{\mathbb{C}}^\lambda}

\newcommand{\ahol}{\mathcal{A}_{\lambda}^2(\Omega)}

\newcommand{\aah}{\mathcal{A}_{\lambda, \text{ah}}^2(\Omega)}

\newcommand{\aph}{\mathcal{A}_{\lambda, \text{ph}}^2(\Omega)}

\newcommand{\acons}{\mathcal{A}_{\mathbb{C}}^2(\Omega)}

\title{Toeplitz operators on pluriharmonic function spaces: Deformation quantization and spectral theory}
\author{Robert Fulsche}

\begin{document}
\maketitle
\begin{abstract}
Quantization and spectral properties of Toeplitz operators acting on spaces of pluriharmonic functions over bounded symmetric domains and $\mathbb C^n$ are discussed. Results are presented on the asymptotics
\begin{align*}
\| T_f^\lambda\|_\lambda &\to \| f\|_\infty\\
\| T_f^\lambda T_g^\lambda - T_{fg}^\lambda\|_\lambda &\to 0\\
\| \frac{\lambda}{i} [T_f^\lambda, T_g^\lambda] - T_{\{f,g\}}^\lambda\|_\lambda &\to 0
\end{align*}
for $\lambda \to \infty$, where the symbols $f$ and $g$ are from suitable function spaces. Further, results on the essential spectrum of such Toeplitz operators with certain symbols are derived.

\medskip
\textbf{AMS subject classification:} Primary: 47B35; Secondary: 30H20, 47A53, 81S10

\medskip
\textbf{Keywords:} Toeplitz operators, pluriharmonic functions, quantization, essential spectrum  
\end{abstract}
\section{Introduction}
Let $\Omega \subset \mathbb C^n$ be a domain and for each $\lambda \in \mathbb R$ sufficiently large let $v_\lambda$ be a probability measure on $\Omega$. Consider the family of Bergman or Segal-Bargmann spaces
\[ \ahol := L^2(\Omega, dv_\lambda) \cap \text{Hol}(\Omega). \]
Each $\ahol$ is known to be a closed subspace of $L^2(\Omega, dv_\lambda)$, hence there exists an orthogonal projection $P^\lambda: L^2(\Omega, dv_\lambda) \to \ahol$. To each $f \in L^\infty(\Omega)$ associate the family of Toeplitz operators
\[ T_f^\lambda: \ahol \to \ahol, \quad T_f^\lambda(g) = P^\lambda(fg). \]
This assignment $f \mapsto T_f^\lambda$ is a common model for quantization, the so-called Toeplitz-quantization. If we consider the derformation quantization in the sense of Rieffel \cite{Rieffel}, the following properties should hold for a sufficiently large class of symbols $f,g$: \
\begin{align}
\lim_{\lambda \to \infty} \| T_f^\lambda\|_\lambda &= \| f\|_\infty \label{eq:1}\\
\lim_{\lambda \to \infty} \| T_f^\lambda T_g^\lambda - T_{fg}^\lambda\|_\lambda &= 0 \label{eq:2} \\
\lim_{\lambda \to \infty} \| \frac{\lambda}{i}[T_f^\lambda, T_g^\lambda] - T_{\{f,g\}}^\lambda\|_\lambda &= 0 \label{eq:3}
\end{align}
Here, we will always assume $\Omega$ to be either $\mathbb C^n$ or a bounded symmetric domain (always considered with the standard weights as discussed below). A lot of work has been done to understand the  quantization properties (\ref{eq:1})-(\ref{eq:3}) in these cases, see e.g. \cite{Bauer_Coburn_Hagger, Bauer_Hagger_Vasilevski2017, Coburn1992, Hagger2018_2} and references therein.

A related question is the spectral theory of Toeplitz operator $T_f^\lambda$ for fixed $\lambda$. If we again assume $\Omega$ to be $\mathbb C^n$ or a bounded symmetric domain, the essential spectrum is well understood: It consists of the boundary values of its symbols (in a certain sense), c.f. \cite{Al-Qabani_Virtanen, Fulsche_Hagger, Hagger2017, Hagger2019} and references therein for the most recent results.

In this work, we investigate these properties in the setting of Hilbert spaces consisting of pluriharmonic functions instead of spaces of holomorphic functions. Toeplitz operators on pluriharmonic function spaces have been studied in a few places, e.g. \cite{Bauer_Furutani, Englis2009}. Yet, many properties still need to be discussed for this setting.

We will analyze both the quantization properties (\ref{eq:1})-(\ref{eq:3}) for a sufficiently large class of symbols and spectral theory for $\VMO_\partial$ symbols. As it turns out (and has already been observed, e.g. in \cite{Englis}) the property (\ref{eq:3}) fails to hold completely (we will repeat the argument for completeness below). Yet, the properties (\ref{eq:1}) and (\ref{eq:2}) hold in the same way as for holomorphic function spaces. For the essential spectrum, we will obtain the same result as for the holomorphic function spaces if the symbol fulfills certain oscillation conditions. Finally, as the quantization property (\ref{eq:3}) fails, pluriharmonic function spaces do not allow for a full quantization procedure. Yet, the other quantization properties (in particular (\ref{eq:2})) have applications of independend interest. We will discuss one of such applications, motivated by results in \cite{Bauer_Hagger_Vasilevski2018, Bauer_Vasilevski2017}.

There are in principle two different approaches to the theory of Toeplitz operators on spaces of pluriharmonic functions. The first one would be to attack the problems directly through hard analysis, possibly immitating proofs from the case of holomorphic function spaces. In this paper, we follow a different idea: Each pluriharmonic function (say, on a simply connected domain) can be written as the sum of a holomorphic and an anti-holomorphic function. This gives rise to a decomposition of the spaces of pluriharmonic functions into the orthogonal sum of two spaces (Bergman spaces of holomorphic and anti-holomorphic functions), which allows us to use established results on Toeplitz operators over holomorphic functions for proving results on Toeplitz operators over pluriharmonic function spaces. The approach also has the advantage that we do not need to distinguish in our proofs between $\Omega = \mathbb C^n$ or $\Omega$ a bounded symmetric domain.

The paper is organized as follows: In Section \ref{sec:prelims}, we settle the basic definitions and recall important results. In Section \ref{sec:quant}, the quantization properties (\ref{eq:1})-(\ref{eq:3}) are studied over pluriharmonic function spaces. Section \ref{sec:spectheory} provides the results on the essential spectrum for pluriharmonic Toeplitz operators with suitable symbols. An application of the quantization property (\ref{eq:2}) in spectral theory is discussed in Section \ref{sec:specdecomp}. Finally, Appendix \ref{appendixA} is added where we provide a result on Toeplitz quantization over the holomorphic Bergman spaces of bounded symmetric domains.
\section{Preliminaries}\label{sec:prelims}
Let $\Omega \subseteq \mathbb{C}^n$ be open and connected. A pluriharmonic function on $\Omega$ is a $C^2$-function $f: \Omega \to \mathbb{C}$ such that
\begin{align*}
\frac{\partial^2 f}{\partial z_j \partial \overline{z}_k} = 0
\end{align*}
for all $j, k = 1, \dots, n$. If $\Omega$ is simply connected one can show that for each pluriharmonic function $f$ on $\Omega$ there are unique holomorphic functions $g, h$ on $\Omega$ with $h(0) = 0$ and
\begin{equation}\label{decompfun}
f = g + \overline{h}.
\end{equation}
We will mainly be concerned with two kinds of domains $\Omega$:
\begin{enumerate}
\item $\Omega = \mathbb{C}^n$,
\item $\Omega$ a bounded symmetric domain in $\mathbb{C}^n$.
\end{enumerate}
The class of bounded symmetric domains includes of course the case where $\Omega = \mathbb B^n$, the open unit ball in $\mathbb C^n$. While we will prove all relevant results on both the unit ball and $\mathbb{C}^n$, we will have to exclude the case of general bounded symmetric domains in some cases - the quantization property (2) for $\VMO$-symbols so far has only been proven in the holomorphic Bergman space setting of $\mathbb B^n$ and not general bounded symmetric domains (cf. Theorem \ref{thm1} below).

On each of these domains, we will consider weighted Hilbert spaces of holomorphic, antiholomorphic or pluriharmonic functions as defined in the following.
\begin{ex}[Segal-Bargmann spaces]
For $\lambda > 0$ let $v_\lambda$ be the measure
\begin{align*}
dv_\lambda(z) = \Big (\frac{\lambda}{\pi} \Big )^n e^{-\lambda |z|^2} dv(z)
\end{align*}
on $\mathbb{C}^n$, where $dv(z)$ is just the usual Lebesgue measure on $\mathbb R^{2n} \cong \mathbb C^n$. $v_\lambda$ is easily seen to be a probability measure. The (holomorphic) Segal-Bargmann spaces $F_\lambda^2(\mathbb{C}^n)$ are the closed subspaces of $L_\lambda^2(\mathbb{C}^n) := L^2(\mathbb{C}^n, dv_\lambda)$ consisting of holomorphic functions. These are reproducing kernel Hilbert spaces with kernels given by
\[ K^\lambda(w,z) = e^{\lambda w \cdot \overline{z}}, \]
where $w \cdot z$ denotes the Euclidean inner product on $\mathbb{C}^n$, being linear in both components.
In an abuse of notation, we will also write $\mathcal{A}_\lambda^2(\mathbb{C}^n)$ instead of $F_\lambda^2(\mathbb{C}^n)$. When we consider a metric on $\mathbb{C}^n$, we mean the usual Euclidean metric
\[ d(z,w) = |z-w|. \]
\end{ex}
\begin{ex}[Bergman spaces on the unit ball]
On $\mathbb{B}^n$ we consider for $\lambda > -1$ the probability measures
\[ dv_\lambda(z) = \frac{\Gamma(n + 1 + \lambda)}{\pi^n \Gamma(\lambda + 1)}(1-|z|^2)^{\lambda} dv(z) \]
Denote by $\mathcal{A}_\lambda^2(\mathbb{B}^n)$ the standard weighted (holomorphic) Bergman space, i.e. the closed subspace of $L_\lambda^2(\mathbb{B}^n) := L^2(\mathbb{B}^n, dv_\lambda)$ consisting of holomorphic functions. Again, $\mathcal{A}_\lambda^2(\mathbb{B}^n)$ is a reproducing kernel Hilbert space with kernel
\[ K^\lambda(w,z) = \frac{1}{(1 - w \cdot \overline{z})^{n + 1 + \lambda}}. \]
We usually consider the unit ball with the metric
\[ d(z,w) = \beta(z,w), \]
$\beta$ being the hyperbolic metric.
\end{ex}
\begin{ex}[Bergman spaces on bounded symmetric domains]
Let $\Omega \subset \mathbb{C}^n$ be a bounded symmetric domain, considered in its Harish-Chandra realization, cf. \cite{Bekolle_Berger_Coburn_Zhu1990, Helgason, Loos, Upmeier}. In particular, $\Omega$ is simply connected (cf. \cite[p. 311]{Helgason}) and contains the origin. Recall that the unit ball $\mathbb{B}^n$ is a particular case of such a bounded symmetric domain, the objects we are going to define below are then the same as already defined for this case.

Denote by $p$ the genus of $\Omega$ and let 
\[ h: \mathbb{C}^n \times \mathbb{C}^n \to \mathbb{C} \]
be the Jordan triple determinant of $\Omega$, which is a certain polynomial holomorphic in the first and anti-holomorphic in the second argument. For $\lambda > p-1$ the measure $v_\lambda$ on $\Omega$ is defined as
\[ dv_\lambda(z) = c_\lambda h(z,z)^{\lambda - p}dv(z), \]
where the constant $c_\lambda$ is chosen such that $v_\lambda$ is a probability measure. $\mathcal{A}_\lambda^2(\Omega)$, the holomorphic Bergman space, is defined as the closed subspace of $L_\lambda^2(\Omega) := L^2(\Omega, dv_\lambda)$ consisting of holomorphic functions. The reproducing kernel of $\mathcal{A}_\lambda^2(\Omega)$ is given by
\[ K^\lambda(w,z) = h(w,z)^{-\lambda}. \]
It is worth mentioning that $K^\lambda(w,0) = 1$ for each $w \in \Omega$. The metric
\[ d(z,w) = \beta(z,w) \]
considered on the bounded symmetric domain is the Bergman distance function $\beta$ obtained from the Riemannian metric with tensor
\begin{equation} \label{metric}
(g_{ij}(z))_{i,j} = \Big ( \frac{\partial^2}{\partial z_i \partial \overline{z}_j} \log K^p(z,z) \Big )_{i,j}.
\end{equation}
\end{ex}
\begin{rmk}
Even in the case of Segal-Bargmann spaces, the metric $d$ is obtained from the Bergman kernel $K^1$ by the formula (\ref{metric}). Since we are going to deal with pluriharmonic function spaces, it is natural to ask whether one should rather define the metric $d$ using the pluriharmonic reproducing kernel (defined below). It turns out that the metric induced by the pluriharmonic Bergman kernel is equivalent to the metric induced by the holomorphic Bergman kernel, hence we may use the usual metric.
\end{rmk}
We will always denote the norm of $L_\lambda^2(\Omega)$ by $\| \cdot\|_\lambda$ and the corresponding inner product by $\langle \cdot, \cdot\rangle_\lambda$. We will also denote by $\| \cdot\|_\lambda$ the operator norm of operators acting on $L_\lambda^2(\Omega)$ or a closed subspace (it will always be clear from the context on which space the operator acts). In contrast, the norm of $L^1(\Omega, dv)$ will be denoted by $\| \cdot \|_{L^1}$. For all the above choices of $\Omega$, we also define the anti-holomorphic and pluriharmonic Bergman spaces (resp. Segal-Bargmann spaces): Define $\aah$ (resp. $\mathcal{A}_{\lambda,\text{ah}}^2(\mathbb{C}^n) := F_{\lambda, \text{ah}}^2(\mathbb{C}^n))$ as the subspace of $L_\lambda^2(\Omega)$ consisting of anti-holomorphic functions and the spaces $\aph$ (resp. $\mathcal{A}_{\lambda,\text{ph}}^2(\mathbb{C}^n) := F_{\lambda, \text{ph}}^2(\mathbb{C}^n))$ as the closed subspaces of $L_\lambda^2(\Omega)$ consisting of pluriharmonic functions. Furthermore, we denote the constant functions by $\acons$. 

There are several relations between these spaces. First of all, observe that there is an isometric $1-1$ correspondence between $\ahol$ and $\aah$ via $f \mapsto \overline{f}$. For each polynomial $p$ in $z = (z_1, \dots, z_n)$ and $q$ in $\overline{z} = (\overline{z}_1, \dots, \overline{z}_n)$ with $q(0) = 0$ it holds
\[ \langle p, q\rangle_{\lambda} = \langle p q^\ast, K^\lambda(\cdot, 0)\rangle_{\lambda} = 0, \]
where $q^\ast$ is the polynomial
\[ q^\ast(z) := \overline{q(\overline{z})}. \] 
Since holomorphic polynomials (resp. anti-holomorphic polynomials) are dense in $\ahol$ (resp. in $\aah$), we obtain an orthogonal direct decomposition
\[ \aph = \ahol \bigoplus \Big (\aah \ominus \acons \Big ). \]
The reproducing kernels of $\aah$ and $\aph$ are therefore given by
\begin{align*}
K_{\text{ah}}^\lambda(w,z) &= \overline{K^\lambda(w,z)} = K^\lambda(z,w),\\
K_{\text{ph}}^\lambda(w,z) &= K^\lambda(w,z) + K_{\text{ah}}^\lambda(w,z) - 1.
\end{align*}
We define the normalized holomorphic reproducing kernel $k^\lambda(w,z)$ for $w,z \in \Omega$ by
\[ k^\lambda(w,z) = \frac{K^\lambda(w,z)}{\| K^\lambda(\cdot, z)\|_{\lambda}} \]
and analogously the normalized anti-holomorphic and pluriharmonic reproducing kernels $k_{\text{ah}}^{\lambda}(w,z)$ and $k_{\text{ph}}^\lambda(w,z)$. The orthogonal projections from $L_\lambda^2(\Omega)$ to $\ahol$, $\aah$, $\aph$ and $\acons$ are denoted by $\phol$, $\pah$, $\pph$ and $\pcons$. They fulfill the relation
\begin{align*}
\pph = \phol + \pah - \pcons.
\end{align*}
We define the holomorphic, anti-holomorphic and pluriharmonic Toeplitz operators with symbol $f \in L^\infty(\Omega)$ by
\begin{align*}
T_f^\lambda &= \phol M_f: \ahol \to \ahol,\\
T_f^{\text{ah}, \lambda} &= \pah M_f: \aah \to \aah,\\
T_f^{\text{ph}, \lambda} &= \pph M_f: \aph \to \aph.
\end{align*}
For each of those Toeplitz operators, the norm can be estimated from above by $\| f\|_\infty$.
The holomorphic and anti-holomorphic Hankel operators with symbol $f \in L^\infty(\Omega)$ are defined as
\begin{align*}
H_f^\lambda &:= (I - \phol)M_f: \ahol \to \ahol^\perp,\\
H_f^{\text{ah}, \lambda} &:= (I - \pah)M_f: \aah \to \aah^\perp.
\end{align*}
For $f \in L^\infty(\mathbb{C}^n)$, they are obviously bounded operators with norm less than $\| f\|_\infty$. Recall that Hankel and Toeplitz operators are related through the relation
\begin{equation}\label{hankelrel}
T_f^\lambda T_g^\lambda - T_{fg}^\lambda = -(H_{\overline{f}}^\lambda)^\ast H_g^\lambda,
\end{equation}
and the analogous relation holds for anti-holomorphic Toeplitz operators.

For a function $f \in L^\infty(\Omega)$ we define the holomorphic, anti-holomorphic and pluriharmonic Berezin transform of $f$ by
\begin{align*}
\mathcal B_\lambda(f)(z) &= \langle f k^\lambda(\cdot, z), k^\lambda(\cdot, z) \rangle_{\lambda}, \quad z \in \Omega,\\
\mathcal B_\lambda^{\text{ah}}(f)(z) &= \langle f k_{\text{ah}}^\lambda (\cdot, z), k_{\text{ah}}^\lambda (\cdot, z)\rangle_{\lambda}, \quad z \in \Omega,\\
\mathcal B_\lambda^{\text{ph}}(f)(z) &= \langle fk_{\text{ph}}^\lambda(\cdot, z), k_{\text{ph}}^\lambda(\cdot, z) \rangle_{\lambda}, \quad z \in \Omega,
\end{align*}
and the pluriharmonic Berezin transform of an operator $A \in \mathcal{L}(\aph)$ by
\begin{align*}
\mathcal B_\lambda^{\text{ph}} A (z) := \langle Ak_{\text{ph}}^\lambda(\cdot, z), k_{\text{ph}}^\lambda(\cdot, z) \rangle_{\lambda}, \quad z \in \Omega.
\end{align*}
In particular, $\mathcal B_\lambda^{\text{ph}} T_f^{\text{ph},\lambda} = \mathcal B_\lambda^{\text{ph}} f$.

We will also need to consider function spaces different from $L^\infty(\Omega)$. By $\UC(\Omega)$ we denote all uniformly continuous (not necessarily bounded) functions on $\Omega$ with respect to the appropriate metric $d$.
For $f \in L_{loc}^1(\Omega)$ define the average of $f$ over the measurable bounded set $E \subset \Omega$ with $|E|>0$ by
\[ f_E = \frac{1}{|E|} \int_E f dv, \]
where $|E|$ denotes the Lebesgue measure of the set. For $z \in \Omega,~ \rho > 0$ set
\[ A_2(f, z, \rho) := \frac{1}{|E(z,\rho)|} \int_{E(z,\rho)} |f - f_{E(z,\rho)}|^2 dv, \]
where $E(z, \rho)$ is the ball with respect to the appropriate metric:
\[ E(z,\rho) = \{ w \in \Omega;~ d(z,w) < \rho\}. \]
Define
\[ \VMO(\Omega) := \{ f: \Omega \to \mathbb{C};~ \lim_{\rho \to 0} A_2(f, z, \rho) = 0 ~\text{uniformly on } \Omega\}, \]
the functions of \textit{vanishing mean oscillation} in the interior, and further 
\[ \VMO_b(\Omega) := \VMO(\Omega) \cap L^\infty(\Omega). \]
For $f$ a bounded and continuous function on $\Omega$ define
\[ \operatorname{Osc}(f)(z) = \sup \{ |f(z) - f(w)|;~ d(z,w) \leq 1 \}, \quad z \in \Omega \]
and for $f \in L^\infty(\Omega)$ set
\[ \operatorname{MO}_\lambda(f,z) = \mathcal B_\lambda(|f|^2)(z) - |\mathcal B_\lambda(f)(z)|^2. \]
The spaces $\VO(\Omega)$ and $\VMO_\partial(\Omega)$ (which is not to be confused with $\VMO_b(\Omega)$) of functions with \textit{vanishing oscillation} and \textit{vanishing mean oscillation} at the boundary are then defined as
\[ \VO(\Omega) = \{ f \in C_b(\Omega); ~ \operatorname{Osc}(f)(z) \to 0,~ d(z,0) \to \infty\}, \]
where $C_b$ denotes the bounded continuous functions, and
\[ \VMO_\partial^\lambda(\Omega) = \{ f \in L^\infty(\Omega); ~ \operatorname{MO}_\lambda(f)(z) \to 0,~ d(z,0) \to \infty \}. \]
Then, denote $\VMO_\partial(\Omega) := \VMO_\partial^p(\Omega)$, where $p$ is the genus of the bounded symmetric domain $\Omega$, or $\VMO_\partial(\mathbb C^n) := \VMO_\partial^1(\mathbb C^n)$. We recall that $\VO(\Omega)$ is contained in $\BUC(\Omega)$, the bounded and uniformly continuous functions, and also in $\VMO_\partial(\Omega)$ \cite{Bauer2005, Bekolle_Berger_Coburn_Zhu1990}.

We will also consider Toeplitz and Hankel operators with symbols in $\UC(\Omega)$. There is a certain dense subspace $\mathcal{D}_\lambda$ of $L_\lambda^2$ (being constructed as a union of a scale of dense subspaces), which is known to be an invariant subspace of $P^\lambda$ and of $M_f$ for each $f \in \UC(\Omega)$ (cf. \cite{Bauer2009, Bauer_Hagger_Vasilevski2017} for details). Hence, it is also an invariant subspace of $P_{\text{ah}}^\lambda$ (since it acts as $P_{\text{ah}}^\lambda(f) =\overline{P^\lambda(\overline{f})}$ and $\mathcal{D}_\lambda$ is closed under complex conjugation) and of $P_{\mathbb{C}}^\lambda$ (since $P_{\mathbb{C}}^\lambda = P^\lambda P_{\text{ah}}^\lambda$). Therefore, Toeplitz operators (resp. anti-holomorphic or pluriharmonic Toeplitz operators) with symbol $f \in \UC(\Omega)$ are considered as densely defined operators
\begin{align*}
T_f^\lambda&: \mathcal{D}_\lambda \cap \ahol \to \mathcal{D}_\lambda\cap \ahol\\
T_f^{\text{ah},\lambda}&: \mathcal{D}_\lambda \cap \aah \to \mathcal{D}_\lambda \cap \aah\\
T_f^{\text{ph},\lambda}&: \mathcal{D}_\lambda \cap \aph \to \mathcal{D}_\lambda \cap \aph
\end{align*}
and can be composed with other Toeplitz operators defined on these dense subspaces.

Toeplitz operators with uniformly continuous symbols are in general unbounded. In contrast, Hankel operators with uniformly continuous symbols, being defined as for bounded symbols, are still bounded, yielding consequences for the semi-commutator of Toeplitz operators with uniformly continuous symbols (using relation (\ref{hankelrel})):
\begin{thm}[\cite{Bauer_Coburn_Hagger, Bauer_Hagger_Vasilevski2017}]\label{thm1}
Assume one of the following:
\begin{enumerate}
\item $\Omega$ a bounded symmetric domain or $\Omega = \mathbb{C}^n$ and $f \in \UC(\Omega)$,
\item $\Omega = \mathbb{B}^n$ or $\Omega = \mathbb{C}^n$ and $f \in \VMO_b(\Omega)$.
\end{enumerate}
Then, $H_f^\lambda$ is bounded with
\begin{align*}
\| H_f^{\lambda} \|_\lambda \to 0
\end{align*} 
as $\lambda \to \infty$. In particular,
\begin{align*}
\| T_f^\lambda T_g^\lambda - T_{fg}^\lambda\|_\lambda \to 0, \quad \lambda \to 0
\end{align*}
holds for any $g \in L^\infty(\Omega)$ or $g \in \UC(\Omega)$.
\end{thm}
As a direct consequence of the above result one obtains the following:
\begin{cor}\label{cor2}
Under the conditions of Theorem \ref{thm1} it holds
\begin{align*}
\| H_f^{\emph{\text{ah}}, \lambda}\|_\lambda \to 0, \quad \lambda \to \infty.
\end{align*}
In particular,
\[ \| T_f^{\emph{\text{ah}}, \lambda} T_g^{\emph{\text{ah}},\lambda} - T_{fg}^{\emph{\text{ah}},\lambda}\|_\lambda \to 0, \quad \lambda \to 0 \]
holds for any $g \in L^\infty(\Omega)$ or $g \in \UC(\Omega)$ as well.
\end{cor}
\begin{proof}
The operator $U: L_\lambda^2(\Omega) \to L_\lambda^2(\Omega),~ f \mapsto \overline{f}$ is an isometric isomorphism of $L_\lambda^2(\Omega)$ and also an isomorphism between $\ahol$ and $\aah$. It holds
\begin{align*}
U H_f^{\text{ah},\lambda} U = H_{\overline{f}}^\lambda
\end{align*}
and hence
\begin{align*}
\| H_f^{\text{ah},\lambda}\|_\lambda = \| U H_f^{\text{ah},\lambda} U\|_\lambda = \| H_{\overline{f}}^\lambda\|_\lambda \to 0
\end{align*}
as $\lambda \to 0$.
\end{proof}
\section{Deformation quantization}\label{sec:quant}
\subsection{The first quantization property}
Since we can decompose $\aph = \ahol \bigoplus \Big ( \aah \ominus \acons \Big )$, the matrix representation of the pluriharmonic Toeplitz operator with respect to this decomposition is
\begin{align}\label{decom}
T_f^{\text{ph}, \lambda} = \begin{pmatrix}[c|c]
T_f^\lambda & A_f^\lambda \\ \hline
B_f^\lambda & T_f^{\text{ah} \ominus \mathbb{C}, \lambda}
\end{pmatrix},
\end{align}
where
\begin{align*}
A_f^\lambda &= \phol M_f: \aah \ominus \acons \to \ahol,\\
B_f^\lambda &= (\pah - \pcons)M_f: \ahol \to \aah \ominus \acons,\\
T_f^{\text{ah} \ominus \mathbb{C}, \lambda} &= (\pah - \pcons) M_f: \aah \ominus \acons \to \aah \ominus \acons.
\end{align*}
\begin{prop}[First quantization property]\label{qp1}
For all $f \in L^\infty(\Omega)$  it holds
\begin{align*}
\| T_f^{\emph{\text{ph}},\lambda}\|_\lambda \to \| f\|_\infty
\end{align*}
as $\lambda \to \infty$.
\end{prop}
\begin{proof}
By the matrix representation above it holds $\| f\|_\infty \geq \| T_f^{\text{ph},\lambda}\|_\lambda \geq \| T_f^\lambda\|_\lambda$. In \cite[Theorem 6.2]{Bauer_Coburn_Hagger} it was proven that $\| T_f^\lambda\|_\lambda \to \| f\|_\infty$ as $\lambda \to \infty$ holds for each $f \in L^\infty(\mathbb C^n)$. We provide the analogous result for $\Omega$ a bounded symmetric domain in Appendix \ref{appendixA}. This completes the proof.
\end{proof}
\begin{rmk}
Let $S_\lambda$ be any family of closed subspaces of $L_\lambda^2(\Omega)$ such that for each sufficiently large $\lambda \in \mathbb R$ it is $\ahol \subseteq S_\lambda \subseteq L_\lambda^2(\Omega)$, e.g. let $S_\lambda$ the space of harmonic functions in $L_\lambda^2$. Then, it follows by the same reasoning that
\[ \| T_f^{S_\lambda}\|_{\lambda} \to \| f\|_\infty, \quad \lambda \to \infty \]
for each $f \in L^\infty(\Omega)$. Here, $T_f^{S_\lambda}$ denotes the Toeplitz operator on $S_\lambda$ with symbol $f$, i.e.
\[ T_f^{S_\lambda}: S_\lambda \to S_\lambda, \quad T_f^{S_\lambda} = P_{S_\lambda} M_f. \]
\end{rmk}
We will prove a related result on the pluriharmonic Berezin transform.
\begin{lmm}\label{prop4}
Let $f \in L^\infty(\Omega)$ be such that $\| H_f^\lambda\|_\lambda, \| H_f^{\emph{\text{ah}}, \lambda}\|_\lambda \to 0$ as $\lambda \to \infty$. Further, let $z \in \Omega$ be such that
\[ \mathcal B_\lambda(f)(z) \to f(z) \]
and
\[ \mathcal B_\lambda^{\emph{\text{ah}}}(f)(z) \to f(z) \]
as $\lambda \to \infty$. Then, it also holds
\begin{align*}
\mathcal B_\lambda^{\emph{\text{ph}}} (f)(z) \to f(z)
\end{align*}
as $\lambda \to \infty$.
\end{lmm}
\begin{proof}
Observe that the result follows trivially for $z = 0$ as
\[ \mathcal B_\lambda^{\text{ph}}(f)(0) = \mathcal B_\lambda(f)(0). \]
Hence, we may assume $z \neq 0$. It is
\begin{align*}
\mathcal B_\lambda^{\text{ph}}(f)(z) = &\frac{1}{\| K^\lambda(\cdot, z) + K_{\text{ah}}^\lambda(\cdot, z) - 1\|_\lambda^2} \\
 &\cdot \langle f(K^\lambda(\cdot, z) + K_{\text{ah}}^\lambda(\cdot, z) - 1), (K^\lambda(\cdot, z) + K_{\text{ah}}^\lambda(\cdot, z) - 1) \rangle_\lambda.
\end{align*}
First, recall that $\| K^\lambda (\cdot, z)\|_\lambda^2 \to \infty$ as $\lambda \to \infty$ for each $z \neq 0$. Observe that, by orthogonality,
\begin{align*}
\frac{1}{\| K^\lambda(\cdot, z) + K_{\text{ah}}^\lambda(\cdot, z) - 1\|_\lambda^2} &= \frac{1}{\| K^\lambda(\cdot, z)\|_\lambda^2 + \| K_{\text{ah}}^\lambda (\cdot, z) - 1\|_\lambda^2}\\
&= \frac{1}{\| K^\lambda(\cdot, z)\|_\lambda^2} \cdot \frac{1}{1 + \frac{\| K_{\text{ah}}^\lambda(\cdot, z) - 1\|_\lambda^2}{\| K^\lambda(\cdot, z)\|_\lambda^2}}.
\end{align*}
It holds
\begin{align*}
\| K_{\text{ah}}(\cdot, z) - 1\|_\lambda^2 = \| K^\lambda(\cdot, z)\|_\lambda^2 - 1
\end{align*}
and therefore
\begin{align*}
\frac{1}{1 + \frac{\| K_{\text{ah}}^\lambda(\cdot, z) - 1\|_\lambda^2}{\| K^\lambda(\cdot,z)\|_\lambda^2}} \to \frac{1}{2}, \quad \lambda \to \infty
\end{align*}
for each $z \neq 0$. We hence need to check the limit only for
\begin{align*}
\frac{1}{2\| K^\lambda(\cdot, z)\|_\lambda^2}\langle f(K^\lambda(\cdot, z) + K_{\text{ah}}^\lambda(\cdot, z) - 1), (K^\lambda(\cdot, z) + K_{\text{ah}}^\lambda(\cdot, z) - 1) \rangle_\lambda.
\end{align*}
By sesquilinearity, we can split this expression into several simpler terms, which we investigate seperately. We first consider those terms which actually contribute to the limit:
\begin{align*}
\frac{1}{2\| K^\lambda(\cdot, z)\|_\lambda^2} \langle fK^\lambda(\cdot, z), K^\lambda(\cdot, z)\rangle_\lambda &= \frac{1}{2}B_\lambda(f)(z) \to \frac{1}{2} f(z) , \quad \lambda \to \infty,\\
\frac{1}{2\| K^\lambda(\cdot, z)\|_\lambda^2} \langle fK_{\text{ah}}^\lambda(\cdot, z), K_{\text{ah}}^\lambda(\cdot, z)\rangle_\lambda &= \frac{1}{2\| K_{\text{ah}}^\lambda(\cdot, z)\|_\lambda^2} \langle fK_{\text{ah}}^\lambda(\cdot, z), K_{\text{ah}}^\lambda(\cdot, z)\rangle_\lambda \\
 &= \frac{1}{2}B_\lambda^{\text{ah}}(f)(z) \to \frac{1}{2} f(z) , \quad \lambda \to \infty.
\end{align*}
Further, since the measure $v_\lambda$ is a probability measure,
\begin{align*}
\frac{1}{\| K^\lambda(\cdot,z)\|_\lambda^2} |\langle f, 1\rangle_\lambda| \leq \frac{1}{\| K^\lambda(\cdot,z)\|_\lambda^2} \| f\|_\infty \to 0, \quad \lambda \to \infty.
\end{align*}
Next, we consider
\begin{align*}
\frac{1}{\| K^\lambda(\cdot,z)\|_\lambda^2}& |\langle fK^\lambda(\cdot,z), K_{\text{ah}}^\lambda(\cdot,z) - 1\rangle_\lambda|\\
&= \frac{1}{\| K^\lambda(\cdot,z)\|_\lambda^2} |\langle (I - P^\lambda)(fK^\lambda(\cdot, z)), K_{\text{ah}}^\lambda(\cdot, z) - 1\rangle_\lambda |\\
&= \frac{1}{\| K^\lambda(\cdot,z)\|_\lambda^2} |\langle H_f^\lambda(K^\lambda(\cdot, z)), K_{\text{ah}}^\lambda(\cdot, z) - 1\rangle_\lambda|\\
&\leq \| H_f^\lambda\|_\lambda \frac{\| K_{\text{ah}}^\lambda(\cdot, z) - 1\|_\lambda}{\| K^\lambda(\cdot, z)\|_\lambda}.
\end{align*}
As already observed above, $\frac{\|K_{\text{ah}}^\lambda(\cdot,z) - 1\|}{\| K^\lambda(\cdot,z)\|}$ converges to 1 as $\lambda \to \infty$. By assumption it holds $\| H_f^\lambda\| \to 0$ as $\lambda \to \infty$ , hence the initial expression converges to 0. The reasoning for
\begin{align*}
\frac{1}{\| K^\lambda(\cdot,z)\|_\lambda^2} \langle fK_{\text{ah}}^\lambda(\cdot,z), K^\lambda(\cdot,z) - 1\rangle_\lambda
\end{align*}
is the same. Finally,
\begin{align*}
\frac{1}{\| K^\lambda(\cdot,z)\|_\lambda^2} |\langle f, K^\lambda(\cdot,z)\rangle_\lambda| \leq \frac{\| f\|_\infty}{\| K^\lambda(\cdot,z)\|_\lambda}
\end{align*}
by the Cauchy-Schwarz inequality, which converges to 0 as $\lambda \to \infty$, and
\begin{align*}
\frac{1}{\| K^\lambda(\cdot,z)\|_\lambda^2} \langle f, K_{\text{ah}}^\lambda(\cdot,z)\rangle_\lambda
\end{align*}
converges in the same way to 0. Putting all these pieces together yields the result.
\end{proof}
\begin{prop}\label{cor5}
For $f \in C_b(\Omega)$ it holds
\[ \mathcal B_\lambda^{\emph{\text{ph}}}(f)(z) \to f(z), \quad \lambda \to \infty \]
for all $z \in \Omega$.
\end{prop}
\begin{proof}
Fix $z \in \Omega$ and let $\varepsilon > 0$ be arbitrary. Let $\delta > 0$ be such that $|f(w) - f(z)| < \varepsilon$ for $w \in E(z,\delta)$. Then,
\begin{align*}
|\mathcal B_\lambda^{\text{ph}}(f)(z) - f(z)| &\leq \int_\Omega |f(w) - f(z)| \frac{|K_{\text{ph}}^\lambda(z,w)|^2}{K_{\text{ph}}^\lambda(z,z)} dv_\lambda(w)\\
&\leq \varepsilon + 2 \| f\|_\infty \int_{\Omega \setminus E(z, \delta)} \frac{|K_{\text{ph}}^\lambda(z,w)|^2}{K_{\text{ph}}^\lambda(z,z)} dv_\lambda(w).
\end{align*}
Let $\chi \in C(\Omega)$ be such that $\chi|_{\Omega \setminus E(z, \delta)} \equiv 1$, $0 \leq \chi \leq 1$ and $\chi(z) = 0$. In particular, $\chi \in \BUC(\Omega)$. By Theorem \ref{thm1} and Corollary \ref{cor2}, $\chi$ fulfills the assumptions of Lemma \ref{prop4} (it is well known that the holomorphic and anti-holomorphic Berezin transforms converge pointwise for such a function), hence
\begin{align*}
\int_{\Omega \setminus E(z, \delta)} \frac{|K_{\text{ph}}^\lambda(z,w)|^2}{K_{\text{ph}}^\lambda(z,z)} dv_\lambda(w) &= \int_{\Omega \setminus E(z, \delta)} [\chi(w) - \chi(z)] \frac{|K_{\text{ph}}^\lambda(z,w)|^2}{K_{\text{ph}}^\lambda(z,z)} dv_\lambda(w)\\
&\leq \int_{\Omega} [\chi(w) - \chi(z)] \frac{|K_{\text{ph}}^\lambda(z,w)|^2}{K_{\text{ph}}^\lambda(z,z)} dv_\lambda(w)\\
&= \mathcal B_\lambda^{\text{ph}}(\chi)(z) - \chi(z) \to 0, \quad \lambda \to \infty.
\end{align*}
Therefore, it holds
\[ \limsup_{\lambda \to \infty} |\mathcal B_\lambda^{\text{ph}} (f)(z) - f(z)| \leq \varepsilon. \]
Since $\varepsilon$ was arbitrarily small the result follows.
\end{proof}
The following result holds for $\Omega = \mathbb C^n$ or $\Omega = \mathbb B^n$:
\begin{prop}
For $f \in \VMO_b(\Omega)$ it holds
\[ \mathcal B_\lambda^{\emph{\text{ph}}}(f)(z) \to f(z),\quad \lambda \to \infty \]
almost everywhere.
\end{prop}
\begin{proof}
This is just a consequence of Lemma \ref{prop4}: It holds $\mathcal B_\lambda(f) \to f$ almost everywhere by \cite[Theorem 6.2]{Bauer_Coburn_Hagger} for $\Omega = \mathbb C^n$ or by Appendix \ref{appendixA} for $\Omega = \mathbb B^n$, the convergence for the anti-holomorphic Berezin transforms follows easily as well. Further, the Hankel operators converge to 0 in norm by Theorem \ref{thm1} and Corollary \ref{cor2}.
\end{proof}
\subsection{The second quantization property}
In what follows, we will also consider the following operators for suitable measurable symbols $f$ in addition to the operators $A_f^\lambda$ and $B_f^\lambda$ defined above:
\begin{align*}
C_f^{1,\lambda} &:= (I- \pph)M_f: \ahol \to \aph^\perp\\
C_f^{2,\lambda} &:= (I - \pph) M_f: \aah \ominus \acons \to \aph^\perp\\
D_f^{1,\lambda} &:= \phol M_f: \aph^\perp \to \ahol\\
D_f^{2,\lambda} &:= (\pah - \pcons)M_f: \aph^\perp \to \aah \ominus \acons\\
E_f^\lambda &:= (\pah - \pcons)M_f: \acons \to \aah \ominus \acons\\
G_f^\lambda &:= \pcons M_f: \aah \ominus \acons \to \acons.
\end{align*}
If $f \in L^\infty(\Omega)$, all those operators are obviously bounded by $\| f\|_\infty$. The following lemma provides all the information on those operators needed for our purposes. During this section, for $f \in \UC(\Omega)$ we always include the case where $\Omega$ is a general bounded symmetric domain, while for $f \in \VMO_b(\Omega)$ we consider only the special case $\Omega = \mathbb{B}^n$. Still, in both cases $\Omega = \mathbb{C}^n$ is allowed.
\begin{lmm}\label{lmm4}
For $f \in \VMO_b(\Omega)$ or $f \in \UC(\Omega)$, the operators $A_f^\lambda$, $B_f^\lambda$, $C_f^{1,\lambda}$, $C_f^{2,\lambda}$, $D_f^{1,\lambda}$, $D_f^{2,\lambda}$, $E_f^\lambda$ and $G_f^\lambda$ are bounded with norm tending to 0 as $\lambda \to \infty$.
\end{lmm}
\begin{proof} 
Observe that
\begin{align*}
B_f^\lambda &= (P_\text{ah}^\lambda - P_\mathbb{C}^\lambda)(I - P^\lambda)M_f = (P_\text{ah}^\lambda - P_\mathbb{C}^\lambda)H_f^\lambda,\\
C_f^{1,\lambda} &= (I - P_\text{ph}^\lambda) (I - P^\lambda)M_f = (I- P_\text{ph}^\lambda)H_f^{\lambda},\\
C_f^{2,\lambda} &= (I-P_\text{ph}^\lambda) (I-P_\text{ah}^\lambda)M_f = (I - P_\text{ph}^\lambda)H_{f}^{\text{ah},\lambda}|_{\aah \ominus \acons},\\
E_f^\lambda &= (P_\text{ah}^\lambda -P_{\mathbb{C}}^\lambda) (I - P^\lambda)M_f = (P_\text{ah}^\lambda - P_{\mathbb{C}}^\lambda) H_f^\lambda|_{\aah \ominus \acons}, 
\end{align*}
which proves the results for those operators using Theorem \ref{thm1} and Corollary \ref{cor2}. Further, $D_f^{1,\lambda} = (C_{\overline{f}}^{1,\lambda})^\ast$, $D_f^{2,\lambda} = (C_{\overline{f}}^{2,\lambda})^\ast$ and $G_f^\lambda = (E_{\overline{f}}^\lambda)^\ast$. Finally,
\begin{align*}
A_f^\lambda &= (P^\lambda - P_\mathbb{C}^\lambda)M_f + P_\mathbb{C}^\lambda M_f = (P^\lambda - P_\mathbb{C}^\lambda)(I - P_{\text{ah}}^\lambda)M_f + P_\mathbb{C}^\lambda M_f\\
&= (P^\lambda - P_\mathbb{C}^\lambda)H_f^{\text{ah}, \lambda}|_{\aah \ominus \acons} + G_f^\lambda,
\end{align*}
which finishes the proof.
\end{proof}
The semi-commutator of two pluriharmonic Toeplitz operators has the matrix representation (with respect to the orthogonal decomposition $\aph = \ahol \bigoplus (\aah \ominus \acons)$)
\begin{align}\label{semicomdec}
T_f^{\text{ph},\lambda} T_g^{\text{ph},\lambda} - T_{fg}^{\text{ph}, \lambda} = \begin{pmatrix}[c|c]
(1,1) & (1,2) \\ \hline
(2,1) & (2,2)
\end{pmatrix},
\end{align}
where
\begin{align*}
(1,1) &= T_f^\lambda T_g^\lambda - T_{fg}^\lambda + A_f^\lambda B_g^\lambda,\\
(1,2) &= T_f^\lambda A_g^\lambda + A_f^\lambda T_g^{\text{ah} \ominus \mathbb{C}, \lambda} - A_{fg}^\lambda,\\
(2,1) &= B_f^\lambda T_g^\lambda + T_f^{\text{ah} \ominus \mathbb{C}, \lambda} B_g^\lambda - B_{fg}^\lambda,\\
(2,2) &= B_f^\lambda A_g^\lambda + T_f^{\text{ah} \ominus \mathbb{C}, \lambda}T_g^{\text{ah} \ominus \mathbb{C}, \lambda} - T_{fg}^{\text{ah} \ominus \mathbb{C}, \lambda}.
\end{align*}
\begin{prop}[Second quantization property]\label{quant2}
Assume $f \in \VMO_b(\Omega)$ or $f \in \UC(\Omega)$. Then, it holds
\[ \| T_f^{\emph{\text{ph}}, \lambda} T_g^{\emph{\text{ph}}, \lambda} - T_{fg}^{\emph{\text{ph}}, \lambda}\|_\lambda \to 0,\quad  \lambda \to \infty \]
for each $g \in L^\infty(\Omega)$ or $g \in \UC(\Omega)$.
\end{prop}
\begin{proof}
We need to show that all four components in equation (\ref{semicomdec}) converge in norm to 0. $\|$(1,1)$\|_\lambda \to 0$ follows by Theorem \ref{thm1} and Lemma \ref{lmm4}. For $\| $(1,2)$\|_\lambda \to 0$ and $\|$(2,1)$\|_\lambda \to 0$, observe that
\begin{align*}
(1,2) = -D_f^{1,\lambda} C_g^{2,\lambda} \text{ and } (2,1) = -D_f^{2,\lambda} C_g^{1,\lambda},
\end{align*}
then use Lemma \ref{lmm4}. Further,  also by Lemma \ref{lmm4}, $\|$(2,2)$\|_\lambda \to 0$ follows if we show 
\[ \| T_f^{\text{ah} \ominus \mathbb{C}, \lambda}T_g^{\text{ah} \ominus \mathbb{C}, \lambda} - T_{fg}^{\text{ah} \ominus \mathbb{C}, \lambda}\|_\lambda \to 0. \]
By Corollary \ref{cor2} it holds
\begin{equation}\label{ahasymp}
\| T_f^{\text{ah}, \lambda} T_g^{\text{ah}, \lambda} - T_{fg}^{\text{ah}, \lambda} \|_\lambda \to 0.
\end{equation}
Using the orthogonal direct decomposition 
\[ \aah = \Big ( \aah \ominus \acons \Big ) \oplus \acons, \]
we get the following matrix representation:
\begin{align}\label{ahdecomp}
T_f^{\text{ah}, \lambda} = \begin{pmatrix}[c|c]
T_f^{\text{ah} \ominus \mathbb{C}, \lambda} & E_f^\lambda \\ \hline
G_f^\lambda & P_\mathbb{C}^\lambda M_f: \acons \to \acons
\end{pmatrix}.
\end{align}
Hence, the matrix representation for $T_f^{\text{ah}, \lambda} T_g^{\text{ah}, \lambda} - T_{fg}^{\text{ah}, \lambda}$ with respect to the same decomposition has the $(1,1)$-entry
\begin{align*}
T_f^{\text{ah} \ominus \mathbb{C}, \lambda} T_g^{\text{ah} \ominus \mathbb{C}, \lambda} - T_{fg}^{\text{ah} \ominus \mathbb{C}, \lambda} + E_f^\lambda G_g^\lambda.
\end{align*}
By equations (\ref{ahasymp}) and (\ref{ahdecomp}) we know that the norm of this operator tends to 0 as $\lambda \to \infty$. Since the norm of $E_f^\lambda G_g^\lambda$ goes to 0 as $\lambda \to \infty$ by Lemma \ref{lmm4},
\[ \| T_f^{\text{ah} \ominus \mathbb{C}, \lambda} T_g^{\text{ah} \ominus \mathbb{C}, \lambda} - T_{fg}^{\text{ah} \ominus \mathbb{C}, \lambda}\|_\lambda \to 0, \quad \lambda \to \infty\]
needs to hold as well.
\end{proof}
\subsection{The third quantization property}
Although the third quantization property holds for a big class of symbols for Toeplitz operators on holomorphic Bergman and Segal-Bargmann spaces, it does not hold on the pluriharmonic spaces, which can be seen by a symmetry argument. This has already been noted in \cite{Englis2006}. We repeat the observation for completeness and give a somewhat refined result. Observe that it holds
\begin{align*}
\overline{\pph h} = \pph \overline{h}
\end{align*}
for $h \in \aph$, as $\pph$ is an integral operator with real-valued kernel. Therefore
\begin{align*}
\overline{T_f^{\text{ph},\lambda} h} = T_{\overline{f}}^{\text{ph},\lambda} \overline{h}
\end{align*}
for each $f \in L^\infty(\Omega)$ and $h \in \aph$. Thus, it holds for $f, g \in L^\infty(\Omega)$ and $h \in \aph$
\begin{align*}
\overline{[T_f^{\text{ph},\lambda}, T_g^{\text{ph},\lambda}]^\ast(h)} = \overline{[T_{\overline{g}}^{\text{ph},\lambda}, T_{\overline{f}}^{\text{ph},\lambda}](h)} = [T_g^{\text{ph},\lambda}, T_f^{\text{ph},\lambda}](\overline{h}) = -[T_f^{\text{ph},\lambda}, T_g^{\text{ph},\lambda}](\overline{h}).
\end{align*}
This implies for the pluriharmonic Berezin transform of $[T_f^{\text{ph},\lambda}, T_g^{\text{ph},\lambda}]$, using that the pluriharmonic reproducing kernel is real-valued:
\begin{align*}
\mathcal B_\lambda^{\text{ph}}([T_f^{\text{ph},\lambda}, T_g^{\text{ph},\lambda}])(z) &= \frac{\langle [T_f^{\text{ph},\lambda}, T_g^{\text{ph},\lambda}] K_\text{ph}^\lambda(\cdot, z),K_\text{ph}^\lambda(\cdot, z) \rangle }{K_\text{ph}^\lambda(z, z)}\\
&= \frac{\langle \overline{[T_f^{\text{ph},\lambda}, T_g^{\text{ph},\lambda}]^\ast K_\text{ph}^\lambda(\cdot, z)},\overline{K_\text{ph}^\lambda(\cdot, z)} \rangle }{K_\text{ph}^\lambda(z, z)}\\
&= -\frac{\langle [T_f^{\text{ph},\lambda}, T_g^{\text{ph},\lambda}] K_\text{ph}^\lambda(\cdot, z),K_\text{ph}^\lambda(\cdot, z) \rangle }{K_\text{ph}^\lambda(z, z)}\\
&= -\mathcal B_\lambda^{\text{ph}}([T_f^{\text{ph},\lambda}, T_g^{\text{ph},\lambda}])(z)
\end{align*}
and therefore $\mathcal B_\lambda^{\text{ph}}([T_f^{\text{ph},\lambda}, T_g^{\text{ph},\lambda}])(z) = 0$ for all $z \in \Omega$. Now let $f, g \in L^\infty(\Omega)$ and $h \in C_b(\Omega)$. Then, if 
\begin{align*}
\Big \| \frac{\lambda}{i}[T_f^{\text{ph},\lambda},T_g^{\text{ph},\lambda}] - T_{h}^{\text{ph},\lambda} \Big \|_\lambda \to 0, \quad \lambda \to \infty
\end{align*}
is assumed to hold, it follows
\begin{align*}
\Big \| \frac{\lambda}{i}[T_f^{\text{ph},\lambda},T_g^{\text{ph},\lambda}] - T_{h}^{\text{ph},\lambda} \Big \|_\lambda &\geq \Big \| \mathcal B_\lambda^{\text{ph}}\Big (\frac{\lambda}{i}[T_f^{\text{ph},\lambda},T_g^{\text{ph},\lambda}] - T_{h}^{\text{ph},\lambda} \Big ) \Big \|_\infty\\
&= \| \mathcal B_\lambda^{\text{ph}} T_{h}^{\text{ph},\lambda}\|_\infty  \geq 0,
\end{align*}
and hence $\| \mathcal B_\lambda^{\text{ph}}(T_{h}^{\text{ph},\lambda})\|_\infty \to 0$, which implies, by Proposition \ref{cor5}, $h = 0$. This gives the following consequence:
\begin{prop}[Third quantization property]
Let $f, g \in L^\infty(\Omega)$ and $h \in C_b(\Omega)$. Then,
\[ \Big \| \frac{\lambda}{i}[T_f^{\emph{\text{ph}},\lambda}, T_g^{\emph{\text{ph}},\lambda}] - T_{h}^{\emph{\text{ph}},\lambda} \Big \|_\lambda \to 0, \quad \lambda \to \infty \]
holds if and only if $h = 0$ and $\|[T_f^{\emph{\text{ph}},\lambda}, T_g^{\emph{\text{ph}},\lambda}]\|_\lambda \in o(1/\lambda)$ as $\lambda \to \infty$.\\
In particular, there cannot be any Poisson structure $\{ \cdot, \cdot\}$ on $\Omega$ such that
\[ \Big \| \frac{\lambda}{i} [T_f^{\emph{\text{ph}},\lambda}, T_g^{\emph{\text{ph}}, \lambda}] - T_{\{ f, g\}}^{\emph{\text{ph}}, \lambda} \Big \|_\lambda \to 0, \quad \lambda \to \infty \]
holds for all $f, g \in C_c^\infty(\Omega)$.
\end{prop}

\section{Spectral theory for \texorpdfstring{$\VMO_\partial$}{VMO} symbols}\label{sec:spectheory}
In this section, we want to find the essential spectrum of $T_f^{\text{ph}, \lambda}$ for fixed $\lambda$ and  $f \in \VMO_\partial(\Omega)$. Here, $\Omega$ is either a general bounded symmetric domain in its Harish-Chandra realization or $\mathbb C^n$. As expected, the essential spectrum consists of the boundary values of the Berezin transform of $f$. The proof is based on standard methods. The main result of this section (Corollary \ref{cor13}) has already been obtained with a different method for the case of the Segal-Bargmann space with $\lambda = 1$ in \cite[Section 4.2]{Bauer_Furutani}.
\begin{lmm}\label{lmmcomp}
Let $f \in C_0(\Omega)$ (i.e. $f$ is continuous and vanishes at the boundary). Then, $T_f^{\emph{\text{ph}}, \lambda}$ is compact.
\end{lmm}
\begin{proof}
First, let $f$ be continuous on $\Omega$ with compact support. Then, $\pph M_{\chi_{\operatorname{supp} f}}$ is a compact operator, hence $T_f^{\text{ph}, \lambda}$ is compact. If $f \in C_0(\Omega)$, take a sequence $f_n$ from $C_c(\Omega)$ which converges to $f$ with respect to $\| \cdot \|_\infty$. Then, $T_{f_n}^{\text{ph}, \lambda}$ converges to $T_f^{\text{ph}, \lambda}$ in norm, hence the operator is also compact.
\end{proof}
\begin{lmm}\label{hankelcomp}
If $f \in \VMO_\partial(\Omega)$, then $H_f^\lambda$ is compact.
\end{lmm}
\begin{proof}
Cf. \cite[Theorem 5.3]{Bauer2005} for the case of Segal-Bargmann spaces with $\lambda = 1$ and \cite[Theorem B]{Bekolle_Berger_Coburn_Zhu1990} for the case of unweighted Bergman spaces on a bounded symmetric domain. The proofs work analogously for the standard weighted cases with general $\lambda$.
\end{proof}
\begin{lmm}\label{compvmo}
If $f \in \VMO_\partial(\Omega)$, then $\mathcal B_\lambda(f) \in C_0(\Omega)$ implies compactness of $T_f^{\emph{\text{ph}}, \lambda}$.
\end{lmm}
\begin{proof}
Consider the matrix representation in equation (\ref{decom}). $H_f^\lambda$ is compact by Lemma \ref{hankelcomp}, hence $A_f^\lambda$ and $B_f^\lambda$, the off-diagonal operators in the matrix representation, are compact by the representations in the proof of Lemma \ref{lmm4}. Compactness of $T_f^\lambda$ follows as usual under the given assumptions (cf. \cite[Theorem 1.1]{Bauer_Isralowitz2012} and \cite[Theorem A] {Englis1999} for more general results on the Segal-Bargmann space and bounded symmetric domains), compactness of $T_f^{\text{ah} \ominus \mathbb{C}, \lambda}$ follows from the decomposition in equation (\ref{ahdecomp}) and the compactness of $T_f^{\text{ah}, \lambda}$.
\end{proof}
\begin{lmm}\label{semicommcomp}
For $f \in \VMO_\partial(\Omega)$ and $g \in L^\infty(\Omega)$, $T_f^{\emph{\text{ph}}, \lambda} T_g^{\emph{\text{ph}}, \lambda} - T_{fg}^{\emph{\text{ph}}, \lambda}$ and $T_g^{\emph{\text{ph}}, \lambda} T_f^{\emph{\text{ph}}, \lambda} - T_{gf}^{\emph{\text{ph}}, \lambda}$ are compact.
\end{lmm}
\begin{proof}
For $f \in \VMO_\partial(\Omega)$, the Hankel operators $H_f^\lambda$ and $H_{\overline{f}}^\lambda$ are compact  (again Lemma \ref{hankelcomp}). By the representations of $T_f^{\text{ph}, \lambda} T_g^{\text{ph}, \lambda} - T_{fg}^{\text{ph}, \lambda}$ and $T_g^{\text{ph}, \lambda} T_f^{\text{ph}, \lambda} - T_{gf}^{\text{ph}, \lambda}$ in Proposition \ref{quant2} and Lemma \ref{lmm4}, the operators are compact.
\end{proof}
\begin{prop}\label{essspec}
Let $f \in \VO(\Omega)$. Then, $T_f^{\emph{\text{ph}}, \lambda}$ is Fredholm if and only if there are constants $R, c > 0$ such that $|f(z)| \geq c$ for all $z \in \Omega$ with $d(z, 0) \geq R$.
\end{prop}
\begin{proof}
First, assume that $|f(z)| \geq c$ for $d(z,0) \geq R$. Let $g$ be continuous with $g(z) = \frac{1}{f(z)}$ for $d(z,0) \geq R$. Then, $g \in L^\infty(\Omega)$ and  $fg - 1$ vanishes on $E(0,R)^c$. In particular, $T_{fg - 1}^{\text{ph}, \lambda}$ is compact by Lemma \ref{lmmcomp}. Therefore, also
\[ T_f^{\text{ph}, \lambda} T_g^{\text{ph}, \lambda} - I = T_f^{\text{ph}, \lambda} T_g^{\text{ph}, \lambda} - T_{fg}^{\text{ph}, \lambda} + T_{fg - 1}^{\text{ph}, \lambda} \]
is compact by Lemma \ref{semicommcomp}. Analogously, $T_g^{\text{ph}, \lambda} T_f^{\text{ph}, \lambda} - I$ is compact. Hence, $T_f^{\text{ph}, \lambda}$ is Fredholm.

On the other hand, assume that $T_f^{\text{ph}, \lambda}$ is Fredholm and that there is a sequence $(z_j)$ in $\Omega$, $z_j \to \partial \Omega$, such that $f(z_j) \to 0$. Since $T_f^{\text{ph}, \lambda}$ is Fredholm, there is a bounded operator $A \in \mathcal{L}(\aph)$ such that $AT_f^{\text{ph}, \lambda} - I$ is compact. Thus,
\[ \| (AT_f^{\text{ph}, \lambda} - I)k^\lambda(\cdot, z_j) \|_\lambda \to 0, \quad j \to \infty, \]
where $k^\lambda$ is the normalized (holomorphic) reproducing kernel, and $k^\lambda(\cdot, z_j)$ converges weakly to $0$ as $j \to \infty$ (even in $\aph$). This implies
\begin{equation}\label{equation7}
\| AT_f^{\text{ph}, \lambda} k^\lambda(\cdot, z_j)\|_\lambda \to 1, \quad j \to \infty.
\end{equation}
We also know that
\[ T_f^{\text{ph}, \lambda} k^\lambda(\cdot, z_j) = T_f^\lambda k^\lambda(\cdot, z_j) + B_f^\lambda k^\lambda(\cdot, z_j) \]
since $k^\lambda(\cdot, z_j) \in \ahol$. $B_f^\lambda$ is compact, hence $\|B_f^\lambda k^\lambda(\cdot, z_j)\|_\lambda \to 0$ for $j \to \infty$. Finally, we will show that $\| T_f^\lambda k^\lambda(\cdot, z_j)\|_\lambda \to 0$ as $j \to \infty$, which will give a contradiction to (\ref{equation7}).

It is
\[ \| T_f^\lambda k^\lambda(\cdot, z_j) \|^2 \leq \langle |f|^2 k^\lambda(\cdot, z_j), k^\lambda(\cdot, z_j) \rangle = \mathcal B_\lambda(|f|^2)(z_j). \]
Since $f \in \VO(\Omega)$, it is $f - \mathcal B_\lambda(f) \in C_0(\Omega)$ \cite{Bauer2005, Bekolle_Berger_Coburn_Zhu1990}, hence $f(z_j) - \mathcal B_\lambda(f)(z_j) \to 0$ and thus $\mathcal B_\lambda(f)(z_j) \to 0$. Therefore, it suffices to show that
\begin{equation}\label{equation8}
\mathcal B_\lambda(|f|^2 - f)(z_j) \to 0.
\end{equation}
But $|f|^2 - f \in \VO(\Omega)$ as well (since $\VO(\Omega)$ is an algebra which is closed under complex conjugation), and $|f(z_j)|^2 - f(z_j) \to 0$, hence (\ref{equation8}) follows.
\end{proof}
\begin{cor}\label{cor13}
For $f \in \VMO_\partial (\Omega)$ it holds
\[ \sigma_{ess}(T_f^{\emph{\text{ph}}, \lambda}) = \mathcal B_{\lambda}(f)(\partial \Omega) := \bigcap_{R > 0} \mathcal B_{\lambda}(f)(E(0,R)^c). \]
If $f$ is even in $\VO(\Omega)$, then
\[ \sigma_{ess}(T_f^{\emph{\text{ph}}, \lambda}) = f(\partial \Omega) := \bigcap_{R > 0} f(E(0,R)^c). \]
\end{cor}
\begin{proof}
The statement for $f \in \VO(\Omega)$ follows directly from the last proposition. If $f \in \VMO_\partial(\Omega)$, then $\mathcal B_\lambda(|f - B_\lambda(f)|^2) \in C_0(\Omega)$ and $\mathcal B_\lambda(f) \in \VO(\Omega)$ (\cite[Theorem B]{Bekolle_Berger_Coburn_Zhu1990} for bounded symmetric domains, \cite[Theorem 5.3]{Bauer2005} for the Segal-Bargmann space). In particular, also $B_\lambda(f - \mathcal B_\lambda(f)) \in C_0(\Omega)$ holds, therefore $T_{f - \mathcal B_\lambda(f)}^{\text{ph}, \lambda}$ is compact by Lemma \ref{compvmo}. Thus,
\[ \sigma_{ess}(T_f^{\text{ph}, \lambda}) = \sigma_{ess}(T_{\mathcal B_\lambda(f)}^{\text{ph}, \lambda}). \]
\end{proof}
\section{Spectral theory through quantization effects}\label{sec:specdecomp}
In \cite{Bauer_Hagger_Vasilevski2018, Bauer_Vasilevski2017}, results on the essential spectra for Toeplitz operators on $\mathcal{A}_\lambda^2(\mathbb{B}^n)$ with symbols of certain product structures were obtained. A crucial tool for this was the fact that the quantization property (\ref{eq:2}) holds for a sufficiently large class of symbols. The aim of this section is to use a similar construction and apply quantization results from Section \ref{sec:quant} to derive spectral results for Toeplitz operators on different Bergman spaces. For simplicity, we will only deal with the case $n=2$ as in \cite{Bauer_Vasilevski2017}, the generalization to $n>2$ follows exactly the computations in \cite{Bauer_Hagger_Vasilevski2018}. Further, we will not deal with symbols of the general product structure allowed in \cite{Bauer_Vasilevski2017}. This has the advantage that we can avoid the use of representation theory to obtain the desired result on the essential spectrum directly. Nevertheless, it is possible without many changes to immitate the representation theoretic constructions to obtain the more general results as in \cite{Bauer_Vasilevski2017}.

Recall that an orthonormal basis for $\mathcal{A}_{\lambda, \text{ph}}^2 (\mathbb B^1)$ is given by the functions
\begin{align*}
e_{a}^\lambda(z) &= \sqrt{\frac{\Gamma(a + \lambda + 2)}{a! \Gamma(\lambda + 2)}} z^{a}, \quad z \in \mathbb B^1, ~ a \in \mathbb{N}_0
\intertext{and}
\overline{e}_{b}^\lambda(z) &= \sqrt{\frac{\Gamma(b + \lambda + 2)}{b! \Gamma(\lambda + 2)}} \overline{z}^b, \quad z \in \mathbb B^1, ~b \in \mathbb{N},
\end{align*}
that is
\[ \mathcal{A}_{\lambda, \text{ph}}^2(\mathbb{B}^1) = \overline{\operatorname{span}}\{ e_a^\lambda; a \in \mathbb{N}_0\} \oplus \overline{\operatorname{span}} \{ \overline{e}_b^\lambda; b \in \mathbb{N} \}. \]
We now introduce the Bergman spaces $\mathcal{A}_{\lambda, \text{ph-h}}^2(\mathbb B^2)$ as the closed subspace of $L_\lambda^2(\mathbb B^2)$ specified by the following orthonormal basis:
\[ \mathcal A_{\lambda, \text{ph-h}}^2(\mathbb B^2) := \overline{\operatorname{span}}\{ \mathfrak{e}_{(a_1,a_2)}^{\lambda,+}, \mathfrak{e}_{(b_1,b_2)}^{\lambda, -}; \quad (a_1, a_2) \in \mathbb{N}_0^2, (b_1, b_2) \in \mathbb{N} \times \mathbb{N}_0 \}\]
Here, the basis functions are defined by
\begin{align*}
\mathfrak{e}_{(a_1,a_2)}^{\lambda,+}(z) &= \sqrt{\frac{\Gamma(a_1 + a_2 + \lambda + 3)}{a_1!a_2! \Gamma(\lambda + 3)}}z_1^{a_1}z_2^{a_2},~ z= (z_1,z_2) \in \mathbb{B}^2, ~(a_1,a_2) \in \mathbb{N}_0^2
\intertext{and}
\mathfrak{e}_{(b_1,b_2)}^{\lambda,-}(z) &= \sqrt{\frac{\Gamma(b_1+b_2 + \lambda + 3)}{b_1!b_2! \Gamma(\lambda + 3)}}\overline{z_1}^{b_1}z_2^{b_2},~ z = (z_1,z_2) \in \mathbb{B}^2, ~ (b_1,b_2) \in \mathbb{N} \times \mathbb{N}_0.
\end{align*}
Thus, $\mathcal A_{\lambda, \text{ph-h}}^2(\mathbb B^2)$ consists of all $C^2$-functions $f$ on $\mathbb B^2$ such that
\begin{align*}
\frac{\partial^2 f}{\partial z_1 \partial \overline{z_1}} &= 0,\\
\frac{\partial f}{\partial \overline{z_2}} &= 0,
\end{align*}
that is, a function $f$ is in $\mathcal A_{\lambda, \text{ph-h}}^2(\mathbb B^2)$ if and only if it is (pluri-)harmonic in $z_1$ and holomorphic in $z_2$ (and square-integrable). In particular, each such function can be written as a power series converging on $\mathbb B^2$:
\begin{align*}
f(z_1,z_2) &= \sum_{j=0}^\infty \Big (\sum_{k = 0}^\infty c_{j,k}z_1^k z_2^j + \sum_{l = 1}^\infty d_{j,l} \overline{z}_1^l z_2^j \Big )\\
&= \sum_{j=0}^\infty \Big (\sum_{k = 0}^\infty c_{j,k}'\mathfrak{e}_{(k,j)}^{\lambda,+}(z) + \sum_{l = 1}^\infty d_{j,l}' \mathfrak{e}_{(l,j)}^{\lambda,-}(z) \Big ).
\end{align*}
Simple calculations yield
\begin{align*}
\mathfrak{e}_{(a_1,a_2)}^{\lambda,+}(z) = e_{a_1}^{a_2 + \lambda + 1}(z_1) e_{a_2}^{\lambda + 1}(z_2)
\intertext{and}
\mathfrak{e}_{(b_1,b_2)}^{\lambda,-}(z) = \overline{e}_{b_1}^{b_2 + \lambda + 1}(z_1) e_{b_2}^{\lambda + 1}(z_2).
\end{align*}
We define for $a_2 \in \mathbb N_0$
\begin{align*}
H_{a_2} := \overline{\text{span}} \{ \mathfrak{e}_{(a_1,a_2)}^{\lambda,+}, \mathfrak{e}_{(b_1,a_2)}^{\lambda,-}; ~ a_1 \in \mathbb{N}_0, b_1 \in \mathbb{N} \}
\end{align*}
and thus get a decomposition
\begin{equation}\label{decompphh}
\mathcal A_{\lambda, \text{ph-h}}^2(\mathbb B^2) = \bigoplus_{a_2 \in \mathbb{N}_0} H_{a_2}.
\end{equation}
One can easily see that each function $f \in H_{a_2}$ can be written in the form
\[ f(z) = f_{a_2}(z_1) e_{a_2}^{\lambda+1}(z_2) \]
for some unique $f_{a_2} \in \mathcal{A}_{a_2 + \lambda + 1, \text{ph}}^2(\mathbb B^1)$. Hence, we can write each function $f \in \mathcal A_{\lambda, \text{ph-h}}^2(\mathbb B^2)$ as a series
\begin{equation}\label{decompmod}
f(z_1,z_2) = \sum_{a_2 \in \mathbb{N}_0} f_{a_2}(z_1) e_{a_2}^{\lambda+1}(z_2)
\end{equation}
for unique $f_{a_2} \in \mathcal{A}_{ a_2 + \lambda + 1, \text{ph}}^2(\mathbb B^1)$ and further have
\[ \| f\|_{\mathcal A_{\lambda, \text{ph-h}}^2(\mathbb B^2)}^2 = \sum_{a_2 \in \mathbb{N}_0} \| f_{a_2}\|_{\mathcal{A}_{a_2 + \lambda + 1}^2(\mathbb B^1)}^2. \]
Letting $u_{a_2}: H_{a_2} \to \mathcal{A}_{a_2 + \lambda + 1, \text{ph}}^2(\mathbb B^1)$ act through
\[ u_{a_2}(f) = f_{a_2} \]
with $f_{a_2}$ the unique coefficient in the series (\ref{decompmod}), we get an isometric isomorphism
\[ U = \bigoplus_{a_2 \in \mathbb{N}_0} u_{a_2}: \mathcal A_{\lambda, \text{ph-h}}^2(\mathbb B^2) = \bigoplus_{a_2 \in \mathbb{N}_0} H_{a_2} \to \bigoplus_{a_2 \in \mathbb{N}_0} \mathcal{A}_{a_2 + \lambda + 1, \text{ph}}^2(\mathbb B^1). \]
For the remaining part of this section, we let $g \in L^\infty (\mathbb B^1)$ and set $\tilde g(z_1,z_2) = g(z_1)$. Let $P_{\text{ph-h}}^\lambda$ be the orthogonal projection $L_\lambda^2(\mathbb{B}^2) \to \mathcal A_{\lambda, \text{ph-h}}^2(\mathbb B^2)$ and consider the Toeplitz operator 
\[ T_{\tilde g}^{\text{ph-h}, \lambda} = P_{\text{ph-h}}^\lambda M_{\tilde g}: \mathcal A_{\lambda, \text{ph-h}}^2(\mathbb B^2) \to \mathcal A_{\lambda, \text{ph-h}}^2(\mathbb B^2). \]
Our last goal will be to prove the following fact:
\begin{prop}\label{propspecdecomp}
Let $g \in \VO(\mathbb B^1)$. Then, $T_{\tilde g}^{\emph{\text{ph-h}}, \lambda}$ is Fredholm if and only if there is some $c>0$ such that $|g(z_1)| \geq c$ for all $z_1 \in \mathbb{B}^1$. In particular, $\sigma_{ess}(T_{\tilde g}^{\emph{\text{ph-h}}, \lambda}) = \overline{g(\mathbb B^1)}$.
\end{prop}
The first step towards achieving this will be the following:
\begin{lmm}
$T_{\tilde{g}}^{\emph{\text{ph-h}}, \lambda}$ acts as
\begin{align*}
\langle T_{\tilde g}^{\emph{\text{ph-h}}, \lambda} \mathfrak{e}_{(a_1,a_2)}^{\lambda,+}, \mathfrak{e}_{(\tilde a_1, \tilde a_2)}^{\lambda, +}\rangle_\lambda &= \begin{cases}
0, ~ & a_2 \neq \tilde a_2\\
\langle T_g^{\emph{\text{ph}}, a_2 + \lambda + 1} e_{a_1}^{a_2 + \lambda + 1}, e_{\tilde a_1}^{a_2 + \lambda + 1}\rangle_{a_2 + \lambda + 1}, ~ & a_2 = \tilde a_2
\end{cases},\\
\langle T_{\tilde g}^{\emph{\text{ph-h}}, \lambda} \mathfrak{e}_{(a_1,a_2)}^{\lambda,+}, \mathfrak{e}_{(b_1,b_2)}^{\lambda, -}\rangle_\lambda &= \begin{cases}
0, ~ & a_2 \neq b_2\\
\langle T_g^{\emph{\text{ph}}, a_2 + \lambda + 1} e_{a_1}^{a_2 + \lambda + 1}, \overline e_{b_1}^{a_2 + \lambda + 1} \rangle_{a_2 + \lambda +1 }, ~ & a_2 = b_2
\end{cases},\\
\langle T_{\tilde g}^{\emph{\text{ph-h}}, \lambda} \mathfrak{e}_{(b_1,b_2)}^{\lambda,-}, \mathfrak{e}_{(a_1,a_2)}^{\lambda, +}\rangle_\lambda &= \begin{cases}
0, ~ & b_2 \neq a_2\\
\langle T_g^{\emph{\text{ph}}, a_2 + \lambda + 1} \overline e_{b_1}^{a_2 + \lambda + 1}, e_{a_1}^{a_2 + \lambda + 1} \rangle_{a_2 + \lambda + 1}, ~ & b_2 = a_2
\end{cases},\\
\langle T_{\tilde g}^{\emph{\text{ph-h}}, \lambda} \mathfrak{e}_{(b_1,b_2)}^{\lambda,-}, \mathfrak{e}_{(\tilde b_1,\tilde b_2)}^{\lambda, -}\rangle_\lambda &= \begin{cases}
0, ~ & b_2 \neq \tilde b_2\\
\langle T_g^{\emph{\text{ph}}, b_2 + \lambda + 1} \overline e_{b_1}^{b_2 + \lambda + 1}, \overline e_{\tilde b_1}^{b_2 + \lambda + 1} \rangle_{b_2 + \lambda + 1}, ~ & b_2 = \tilde b_2
\end{cases}.
\end{align*}
In particular, $T_{\tilde g}^{\emph{\text{ph-h}}, \lambda}$ leaves the decomposition (\ref{decompphh}) invariant.
\end{lmm}
\begin{proof}
The computations are identical to those in the proof of \cite[Lemma 2.2]{Bauer_Vasilevski2017}. We reproduce them to prove the first identity here, the remaining three cases can be deduced using the same calculations.

Let $(a_1,a_2), (\tilde a_1, \tilde a_2) \in \mathbb N_0^2$. Then,
\begin{align*}
\langle &T_{\tilde g}^{\text{ph-h}, \lambda} \mathfrak e_{(a_1,a_2)}^{\lambda,+}, \mathfrak e_{(\tilde a_1, \tilde a_2)}^{\lambda,+}\rangle\\
 &= \sqrt{\frac{\Gamma(a_1 + a_2 + \lambda + 3)}{a_1! a_2! \Gamma(\lambda + 3)} \frac{\Gamma(\tilde a_1 + \tilde a_2 + \lambda + 3)}{\tilde a_1! \tilde a_2! \Gamma(\lambda + 3)}} \frac{\Gamma(\lambda + 3)}{\pi^2\Gamma(\lambda +1)}\\
& \times \int_{\mathbb B^2} g(z_1) z_1^{a_1} z_2^{a_2} \overline{z_1}^{\tilde a_1} \overline{z_2}^{\tilde a_2} (1-(|z_1|^2 + |z_2|^2))^\lambda dv(z_1,z_2).
\intertext{Introducing polar coordinates $z_1 = r_1e^{i\theta_1}, z_2 = r_2e^{i\theta_2}$, we obtain}
&= \sqrt{\frac{\Gamma(a_1 + a_2 + \lambda + 3) \Gamma(\tilde a_1 + \tilde a_2 + \lambda + 3)}{a_1! a_2! \tilde a_1 ! \tilde a_2!}} \frac{1}{\pi^2\Gamma(\lambda + 1)} \int_0^{2\pi} e^{i\theta_2(a_2 - \tilde{a}_2)} d\theta_2\\
&\quad \times \int_{ \{r_1, r_2 > 0; r_1^2 + r_2^2 < 1 \}} \int_0^{2\pi} g(r_1e^{i\theta_1}) r_1^{a_1 + \tilde a_1 + 1} r_2^{a_2 + \tilde a_2 + 1} e^{i \theta_1(a_1 - \tilde a_1)} \\
&\quad \times (1-r_1^2 - r_2^2)^\lambda d\theta_1 dr_2 dr_1 .
\intertext{Of course, the first integral in this expression equals $0$ for $a_2 \neq \tilde a_2$ and $2\pi$ for $a_2 = \tilde a_2$. For the latter case, we get}
&= \sqrt{\frac{\Gamma(a_1 + a_2 + \lambda + 3) \Gamma(\tilde a_1 + a_2+ \lambda + 3)}{a_1! (a_2!)^2 \tilde a_1 !}} \frac{2}{\pi \Gamma(\lambda + 1)}\\
&\quad \times  \int_{\{ r_1, r_2 > 0; r_1^2 + r_2^2 < 1\} } \int_0^{2\pi} g(r_1e^{i\theta_1}) r_1^{a_1 + \tilde a_2 + 1} r_2^{2a_2 + 1}e^{i  \theta_1 (a_1 - \tilde a_1)}\\
&\quad \times (1- r_1^2 - r_2^2)^\lambda d\theta_1 dr_2 dr_1.
\intertext{Using the substitution $s = \frac{r_2}{\sqrt{1-r_1^2}}$ in the $r_2$ integral we get }
&= \sqrt{\frac{\Gamma(a_1 + a_2 + \lambda + 3) \Gamma(\tilde a_1 + a_2 + \lambda + 3)}{a_1! (a_2!)^2 \tilde a_1 !}} \frac{2}{\pi \Gamma(\lambda + 1)}\\
&\quad \times \int_0^{2\pi} \int_0^1 g(r_1e^{i\theta_1}) r_1^{a_1 + \tilde a_1 + 1}e^{i \theta_1 (a_1 - \tilde a_1)} (1-r_1^2)^{a_2 + \lambda + 1}  d\theta_1 dr_1 \\
&\quad \times \int_0^1 s^{2a_2 + 1} (1-s^2)^\lambda ds\\
&= \sqrt{\frac{\Gamma(a_1 + a_2 + \lambda + 3) \Gamma(\tilde a_1 + a_2 + \lambda + 3)}{a_1! (a_2!)^2 \tilde a_1 !}} \frac{1}{\pi \Gamma(\lambda + 1)} \\
&\quad  \times \int_{\mathbb B} g(z_1) z_1^{a_1} \overline z_1^{\tilde a_1} (1 - |z_1|^2)^{a_2 + \lambda + 1} dv(z) \int_0^1 s^{a_2} (1-s)^\lambda ds.
\intertext{Using the beta function $B(x,y) = \int_0^1 s^{x-1}(1-s)^{y-1} ds$ and the well-known identity $B(x,y) = \dfrac{\Gamma(x)\Gamma(y)}{\Gamma(x+y)}$ we obtain}
&= \frac{\Gamma(a_2 + \lambda + 2)}{a_2! \Gamma(\lambda + 1)} B(a_2 + 1, \lambda + 1)\\
&\quad \times \int_{\mathbb B} g(z_1) e_{a_1}^{a_2 + \lambda + 1}(z_1) \overline{e}_{\tilde a_1}^{a_2 + \lambda + 1}(z_1) \frac{\Gamma(a_2 + \lambda + 3)}{\pi \Gamma(a_2 + \lambda + 2)} (1 - |z_1|^2)^{a_2 + \lambda + 1} dvz\\
&= \langle T_g^{a_2 + \lambda + 1} e_{a_1}^{a_2 + \lambda + 1}, e_{\tilde a_1}^{a_2 + \lambda + 1}\rangle_{a_2 + \lambda + 1}.
\end{align*}
\end{proof}
With the isometry $U$ introduced above we obtain:
\begin{cor}
It holds
\begin{align*}
&T_{\tilde g}^{\emph{\text{ph-h}}, \lambda}: \mathcal A_{\lambda, \emph{\text{ph-h}}}^2(\mathbb B^2) \to \mathcal A_{\lambda, \emph{\text{ph-h}}}^2(\mathbb B^2)\\
\cong &\bigoplus_{a_2 \in \mathbb N_0} T_g^{\emph{\text{ph}}, a_2 + \lambda + 1}: \bigoplus_{a_2 \in \mathbb N_0} \mathcal A_{a_2 + \lambda + 1, \emph{\text{ph}}}^2(\mathbb B^1) \to \bigoplus_{a_2 \in \mathbb N_0} \mathcal A_{a_2 + \lambda + 1, \emph{\text{ph}}}^2(\mathbb B^1). 
\end{align*}
\end{cor}
For proving Proposition \ref{propspecdecomp}, we will also need the following well known fact.
\begin{lmm}\label{fredholmlemma}
Let $H_k$, $k \in \mathbb{N}_0$ be a family of Hilbert spaces and let $\bigoplus_k H_k$ denote their direct orthogonal sum. For a family of operators $A_k \in \mathcal L(H_k)$ let $A := \bigoplus_k A_k$ act diagonally on $H$. Then, $A$ is Fredholm if and only if each $A_k$ is Fredholm and there are Fredholm regularizers $B_k^1, B_k^2$ of $A_k$ such that
\[ \| A_k B_k^1 - I\|_{\mathcal L(H_k)} \to 0, \quad k \to \infty \]
and
\[ \| B_k^2 A_k - I \|_{\mathcal L(H_k)} \to 0, \quad k \to \infty. \] 
\end{lmm}
\begin{proof}[Proof of Proposition \ref{propspecdecomp}]
First, assume $|g(z_1)| \geq c > 0$ for all $z_1 \in \mathbb{B}^1$. Then, $T_g^{\text{ph}, a_2 + \lambda + 1} \in \mathcal L(\mathcal A_{a_2 + \lambda + 1, \text{ph}}^2 (\mathbb B^1))$ is Fredholm by Proposition \ref{essspec}. Further, it holds
\begin{align*}
\| T_g^{\text{ph}, a_2 + \lambda + 1} T_{1/g}^{\text{ph}, a_2 + \lambda + 1} - I \|_{a_2 + \lambda + 1} \to 0, \quad a_2 \to \infty
\end{align*}
by Proposition \ref{quant2} (recall that $\VO(\mathbb B^1)$ is contained in $\UC(\mathbb B^1)$) and also
\[ \| T_{1/g}^{\text{ph}, a_2 + \lambda + 1} T_g^{\text{ph}, a_2 + \lambda + 1} - I\|_{a_2 + \lambda + 1} \to 0, \quad a_2 \to \infty. \] 
Hence, $T_g^{\text{ph-h}, \lambda}$ is Fredholm by Lemma \ref{fredholmlemma}.

On the other hand, assume that $\inf_{z_1 \in \mathbb{B}^1} |g(z_1)| = 0$. There are two cases:
\begin{enumerate}
\item There is a sequence $(z_1^j)_j \in \mathbb B^1$ with $z_1^j \to \partial \mathbb B^1$ such that $g(z_1^j) \to 0$,
\item there is some $z_1 \in \mathbb{B}^1$ such that $g(z_1) = 0$.
\end{enumerate}
In the first case, the operators $T_g^{\text{ph}, a_2 + \lambda + 1}$ on $\mathcal A_{a_2 + \lambda + 1, \text{ph}}^2(\mathbb B^1)$ are not Fredholm by Proposition \ref{essspec}, hence $T_{\tilde g}^{\text{ph-h}, \lambda}$ cannot be Fredholm by Lemma \ref{fredholmlemma}. In the second case, observe the following: Consider the sequence $(f_j)_j \subset \mathcal A_{\lambda, \text{ph-h}}^2(\mathbb B^2)$ defined on the decomposition (\ref{decompphh}) via
\begin{align*}
f_j = \big (\delta_{a_2, j} k^{a_2 + \lambda + 1}(\cdot, z_1) \big )_{a_2 \in \mathbb{N}_0} \in \bigoplus_{a_2 \in \mathbb N_0} \mathcal{A}_{a_2 + \lambda + 1}^2(\mathbb B^1),
\end{align*}
where $k^{a_2 + \lambda + 1}$ is the normalized reproducing kernel on $\mathcal A_{a_2 + \lambda + 1}^2(\mathbb B^1)$. In particular, $f_j \to 0$ weakly as it is an orthonormal sequence. Then,
\begin{align*}
\| T_{\tilde g}^{\text{ph-h},\lambda} f_j\|_\lambda^2 &\leq \langle g k^{j+\lambda+1}(\cdot, z_1), g k^{j + \lambda + 1}(\cdot, z_1)\rangle_{j +\lambda + 1}\\
&= \langle |g|^2 k^{j + \lambda + 1}(\cdot, z_1), k^{j +\lambda + 1}(\cdot, z_1) \rangle_{j +\lambda + 1}\\
&= \mathcal B_{j + \lambda + 1}(|g|^2)(z_1),
\end{align*}
which denotes the (holomorphic) Berezin transform of $|g|^2$ on $\mathbb B^1$. Since $g$ is assumed to be in $\VO(\mathbb B^1)$, it holds in particular $|g|^2 \in C_b(\mathbb B^1)$. Hence,
\[ \mathcal B_{j + \lambda + 1}(|g|^2)(z_1) \to |g|^2(z_1) = 0, \quad j \to \infty. \]
But this means that $(T_{\tilde g}^{\text{ph-h}, \lambda} f_j)_{j \in \mathbb{N}}$ converges strongly to zero. Hence, $T_{\tilde g}^{\text{ph-h}, \lambda}$ cannot be Fredholm, as no Fredholm operator can map a weakly convergent zero sequence (which is not already strongly convergent) to a strongly convergent zero sequence.
\end{proof}

\appendix
\section{The limit of the norm of Toeplitz operators on bounded symmetric domains}\label{appendixA}
In this section we are going to provide a proof of the following fact for $\Omega$ a bounded symmetric domain in its Harish-Chandra realization:
\begin{prop}\label{unitballberezin}
Let $f \in L^\infty(\Omega)$. Then it holds
\[ \mathcal B_\lambda(f) \to f ~ a.e.,~\lambda \to \infty \]
and also
\[ \lim_{\lambda \to \infty} \| \mathcal B_\lambda(f) \|_\infty = \lim_{\lambda \to \infty} \| T_f^\lambda\|_\lambda = \| f\|_\infty. \]
\end{prop}
The corresponding result for the Segal-Bargmann spaces was first proven in \cite{Bauer_Coburn_Hagger}. The proof here is heavily motivated by the Segal-Bargmann space proof. The main technical difference is the fact that we need to conclude the proof first locally around $0$ and ``patch things together'' afterwards, instead of proving it globally right away. This modification of the proof is necessary due to the fact that the Hardy-Littlewood maximal function $f^\ast$ of $f \in L^\infty(\mathbb C^n)$ behaves well under certain automorphisms of $\mathbb{C}^n$, namely shifts  (i.e. $f^\ast(w) = \big (f(\cdot - w)\big )^\ast(0)$), but the corresponding property fails with respect to the geodesic symmetries of bounded symmetric domains.

Before attempting the proof, we need to recall a few more facts on bounded symmetric domains in addition to those mentioned in the beginning.

Let $\Omega \subset \mathbb C^n$ be a bounded symmetric domain in its Harish-Chandra realization. Let $\operatorname{Aut}(\Omega)$ denote the group of holomorphic automorphisms of $\Omega$ and $\operatorname{Aut}_0(\Omega)$ the connected component containing the identity. Denote by $K$ the maximal subgroup of $\operatorname{Aut}_0(\Omega)$ stabilizing 0. If $r$ denotes the rank of $\Omega$, there are elements $e_1, \dots, e_r \in \mathbb{C}^n$ of $\mathbb R$-linearly independend vectors such that each $z \in \mathbb C^n$ can be written in the form
\[ z = k\sum_{j=1}^r t_j e_j \]
for some $k \in K$ and $t_1 \geq t_2 \geq \dots \geq t_r \geq 0$. Further,
\[ z \mapsto \| z\|_{\Omega} := t_1 \]
is well defined and a norm on $\mathbb C^n$, the \textit{spectral norm} of $\Omega$ (cf. \cite[p. 64]{Upmeier}) and it holds
\[ \Omega = \{ z \in \mathbb{C}^n;~ \| z\|_{\Omega} < 1\}. \]
Further, the Jordan triple determinant $h$ is given on the diagonal by the formula
\[ h(z,z) = \prod_{j=1}^r (1-t_j^2). \]
Finally, for $z \in \Omega$ we denote by $\varphi_z$ the geodesic symmetry interchanging $z$ and $0$. 

For a function $f \in L^\infty(\Omega)$ denote by $\tilde{f}$ the continuation of $f$ to $\mathbb{C}^n$ by zero. By $f^\ast$ we denote the Hardy-Littlewood maximal function of $\tilde{f}$, which is defined on $\mathbb C^n$ by
\[ f^\ast(w):= \sup_{\rho > 0} \frac{1}{|B(w,\rho)|} \int_{B(w,\rho)} |\tilde{f}(z)| dv(z). \] 
Here, $B(w,\rho)$ denotes the Euclidean ball around $w$ with radius $\rho$ and $|B(w,\rho)|$ denotes the volume of the ball.
\begin{lmm}
There is a constant $C > 0$ such that for each $f \in L^\infty (\Omega)$ and all $\lambda \geq p+1$ it holds
\[ |\mathcal B_\lambda(f)(0)| \leq C f^\ast (0). \]
\end{lmm}
\begin{proof} For each $\lambda \geq p+1$ let $m_\lambda$ be the smallest integer such that
\[ \Omega \subseteq B \Big (0,\sqrt{\frac{m_\lambda}{\lambda - p}} \Big ). \]
Writing $\mathbb{C}^n = \bigcup_{m=1}^\infty B\big (0, \sqrt{m/(\lambda - p)} \big ) \setminus B\big (0, \sqrt{(m-1)/(\lambda - p)} \big )$, one gets (using that $\tilde{f} = 0$ outside $\Omega$)
\begin{align*}
|\mathcal B_\lambda(f)(0)| &\leq c_\lambda \int_{\Omega} |f(z)|h(z,z)^{\lambda - p} dv(z)\\
&= c_\lambda \sum_{m=1}^{m_\lambda} \int_{B\big (0, \sqrt{\frac{m}{\lambda-p}}\big ) \setminus B\big (0, \sqrt{\frac{m-1}{\lambda-p}}\big )} |\tilde{f}(z)|h(z,z)^{\lambda - p} dv(z).
\end{align*}
Since the norms $\| \cdot\|_\Omega$ and $| \cdot |$ are equivalent, there is some $c' > 0$ such that
\[ \| z\|_\Omega^2 \geq c' |z|^2 \]
holds true for all $z \in \mathbb C^n$. Therefore, for $m = 1, \dots, m_\lambda$ and $z \in \Omega$ with $\frac{m-1}{\lambda - p} \leq |z|^2 < \frac{m}{\lambda - p}$ it holds
\begin{align*}
0 \leq h(z,z) = \prod_{j=1}^r (1 - t_j^2) \leq 1-t_1^2 = 1 - \| z\|_\Omega^2 \leq 1 - c'|z|^2 \leq 1 - c'\frac{m-1}{\lambda - p}.
\end{align*}
We obtain the following estimate, using again that $\tilde{f} = 0$ outside $\Omega$:
\begin{align*}
|\mathcal B_\lambda(f)(0)| &\leq c_\lambda \sum_{m=1}^{m_\lambda} \int_{B\big (0, \sqrt{\frac{m}{\lambda-p}}\big ) \setminus B\big (0, \sqrt{\frac{m-1}{\lambda-p}}\big )} |\tilde f(z)| \Big (1-c'\frac{m-1}{\lambda-p} \Big )^{\lambda-p} dv(z)\\
&\leq c_\lambda \sum_{m=1}^{m_\lambda} \Big ( 1 - c'\frac{m-1}{\lambda-p}\Big )^{\lambda-p} \int_{B\big (0, \sqrt{\frac{m}{\lambda-p}}\big )} |\tilde f(z)| dv(z)\\
&= \frac{c_\lambda n!}{\pi^n (\lambda - p)^n} \sum_{m=1}^{m_\lambda} m^n\Big (1-c'\frac{m-1}{\lambda-p}\Big )^{\lambda-p} \\
&\quad \quad \cdot \frac{1}{|B(0, \sqrt{m/(\lambda-p)})|} \int_{B\big (0, \sqrt{\frac{m}{\lambda-p}}\big )} |\tilde f(z)| dv(z)\\
&\leq f^\ast(0)\frac{c_\lambda n!}{\pi^n(\lambda - p)^n} \sum_{m=1}^{m_\lambda} m^n \Big ( 1 - c'\frac{m-1}{\lambda-p} \Big )^{\lambda-p}.
\end{align*}
As $(1 - c'\frac{m-1}{\lambda-p})^{\lambda-p} \leq e^{-c'(m-1)}$, it follows
\begin{align*}
|\mathcal B_\lambda(f)(0)| &\leq f^\ast(0) \frac{c_\lambda n!}{\pi^n(\lambda - p)^n} \sum_{m=1}^{\infty} m^ne^{-c'(m-1)}.
\end{align*}
This series is of course convergent. The coefficient $\frac{c_\lambda}{\pi^n(\lambda - p)^n}$ remains bounded as $\lambda \to \infty$ since $c_\lambda \sim \lambda^n$, which can be seen from an explicit formula for $c_\lambda$ contained in \cite{Faraut_Koranyi}.
\end{proof}
\begin{lmm}\label{lmm21}
There exists a constant $C' > 0$, independend of $\lambda \geq p+1$, such that for each $f \in L^\infty(\Omega)$ it holds
\[ |\mathcal B_\lambda(f)(z)| \leq C' f^\ast (z)\]
on a neighbourhood of $0$.
\end{lmm}
\begin{proof}
For $f = 0$ this is trivial. Otherwise, it holds $f^\ast(0) > 0$ by the definition of $f^\ast$ and the result follows from the previous lemma, continuity of $\mathcal B_\lambda(f)$ and lower semicontinuity of $f^\ast$, i.e. the fact that
\[ \{ z \in \mathbb C^n; ~ f^\ast(z) > \frac{1}{2C} |\mathcal B_\lambda(f)(0)| \} \]
is open (with $C$ from the previous lemma).
\end{proof}
\begin{lmm}
For $f \in L^\infty(\Omega)$ it holds $\mathcal B_\lambda f\to f$ almost everywhere on a neighbourhood of $0$.
\end{lmm}
\begin{proof}
Let $\varepsilon, \delta > 0$ and further let $g \in C_b(\overline{\Omega})$ such that $\| f - g\|_{L^1} < \delta$. Take $O_{f-g}$ to be the neighbourhood of 0 obtained from Lemma \ref{lmm21} applied to the function $f-g$. We are going to prove that
\[ \{ w \in O_{f-g};~ \limsup_{\lambda \to \infty} |\mathcal B_\lambda(f)(w) - f(w)| > \varepsilon \} \]
is a set of measure zero. It holds
\begin{align*}
|\mathcal B_\lambda(f)(w) &- f(w)|\\
 & \quad \leq |\mathcal B_\lambda(f)(w) - \mathcal B_\lambda(g)(w)| + |\mathcal B_\lambda(g)(w) - g(w)| + |g(w) - f(w)|.
\end{align*}
As $g$ is uniformly continuous, it holds $\mathcal B_\lambda(g) \to g$ uniformly as $\lambda \to \infty$, hence
\begin{align*}
\limsup_{\lambda \to \infty} |\mathcal B_\lambda(f)(w) - f(w)| &\leq \limsup_{\lambda \to \infty} |\mathcal B_\lambda(f)(w) - \mathcal B_\lambda(g)(w)| + |g(w) - f(w)|.
\end{align*}
By Markov's inequality,
\[ |\{ w \in \Omega; ~ |g(w) - f(w)| > \varepsilon\} | \leq \frac{\| g-f\|_{L^1}}{\varepsilon} \leq \frac{\delta}{\varepsilon}. \]
Further, it holds for $z \in O_{f-g}$
\begin{align*}
|\mathcal B_\lambda(f)(w) - \mathcal B_\lambda(g)(w)| = |\mathcal B_\lambda(f-g)(w)| \leq C'(f-g)^\ast(w).
\end{align*}
By the weak $(1,1)$-inequality for the Hardy-Littlewood maximal function, there exists $C_1 > 0$ independend of $\varepsilon, \delta$ such that
\begin{align*}
|\{ w \in \Omega;~ (f-g)^\ast(w) > C'\varepsilon \}| \leq \frac{C_1}{C'\varepsilon}\| f-g\|_{L^1} \leq \frac{C_1\delta}{C'\varepsilon}.
\end{align*}
Setting everything together, we obtain
\begin{align*}
|\{ w \in O_{f-g};~ \limsup_{\lambda \to \infty} |\mathcal B_\lambda(f)(w) - f(w)| > \varepsilon \}| \leq \frac{\delta}{\varepsilon} \Big ( 1 + \frac{C_1}{C'} \Big ).
\end{align*}
As $\delta > 0$ was arbitrary, it follows that the set is a zero set for each $\varepsilon > 0$.
\end{proof}

\begin{proof}[Proof of Proposition \ref{unitballberezin}]
By the previous lemma, $\mathcal B_\lambda(f) \to f$ on a neighbourhood of zero for arbitrary $f \in L^\infty(\Omega)$. Therefore, for any $z \in \Omega$, it holds
\[ \mathcal B_\lambda(f \circ \varphi_z)(w)  \to f\circ \varphi_z(w) \]
almost everywhere on a neighbourhood of $0$. As the Berezin transform is invariant under composition with the $\varphi_z$, it follows
\[ \mathcal B_\lambda(f \circ \varphi_z)(w) = \mathcal B_\lambda(f)(\varphi_z(w)) \to f(\varphi_z(w)) \]
almost everywhere on a zero neighbourhood, hence for each $z \in \Omega$ there exists a neighbourhood $O_z$ of $z$ such that
\[ \mathcal B_\lambda(f)(w) \to f(w) \]
almost everywhere on $O_z$. $\{ O_z\}_{z \in \Omega}$ is an open cover of $\Omega$, hence has a countable subcover. As the union of countably many zero sets is  still a zero set, it follows that $\mathcal B_\lambda(f) \to f$ almost everywhere on the whole of $\Omega$.

It remains to prove
\[ \| \mathcal B_\lambda f\|_\infty \to \| f\|_\infty, \quad \lambda \to \infty, \]
which is identical to the case of the Segal-Bargmann space. Let $\varepsilon > 0$. By Egorov's Theorem, we can choose a set $A_\varepsilon \subseteq \Omega$ such that $|A_\varepsilon| > 0$, $|f(z)| \geq \| f\|_\infty - \varepsilon$ for $z \in A_\varepsilon$ and $\mathcal B_\lambda(f) \to f$ uniformly on $A_\varepsilon$. Recall that $|\mathcal B_\lambda(f)(z)| \leq \| T_f\|_\lambda \leq  \| f\|_\infty$ holds for all $z \in \Omega$. Then,
\begin{align*}
\| f\|_\infty &\geq \limsup_{\lambda \to \infty} \| \mathcal B_\lambda(f)\|_\infty \geq \liminf_{\lambda \to \infty} \| \mathcal B_\lambda(f)\|_\infty\\
&\geq \liminf_{\lambda \to \infty} \| \mathcal B_\lambda(f)|_{A_\varepsilon}\|_\infty \geq \| f\|_\infty - \varepsilon.
\end{align*}
\end{proof}

\section*{Acknowledgement}
The author wants to thank Wolfram Bauer for his help and support and Raffael Hagger for valuable discussions.

\bibliographystyle{amsplain}
\bibliography{References}

\end{document}